\documentclass[oneside,11pt,reqno]{amsart}
\usepackage[margin = .8 in]{geometry}
\usepackage{amssymb}
\usepackage{graphicx, caption, subcaption} 
\usepackage{bbm}
\usepackage{setspace}
\usepackage{mathrsfs}
\usepackage{tensor}
\usepackage{tikz}
\usepackage{tikz-cd}
\usepackage{epstopdf}
\usepackage{enumitem}
\usepackage{blkarray}
\usepackage{nicematrix}
\usepackage[enableskew]{youngtab}
\usepackage{young}
\usepackage[all,cmtip]{xy}
\usepackage{color,polynom}
\usepackage{hyperref}
\usepackage{mathtools}
\usepackage{mathdots}
\usepackage{float}
\usepackage{xcolor}
\usepackage[dvipsnames]{xcolor}

\usepackage{mdframed}
\usepackage{lipsum}

\newcommand{\zz}{\ensuremath{\mathbb{Z}}}

\newcommand{\nn}{\ensuremath{\mathbb{N}}}

\newcommand{\mI}{\ensuremath{\mathcal{I}}}
\newcommand{\mC}{\ensuremath{\mathcal{C}}}

\newcommand{\id}{\ensuremath{\text{id}}}

\newcommand{\im}{\operatorname{im}}

\theoremstyle{definition}

\newmdtheoremenv{frm-thm}{Theorem}
\newmdtheoremenv{frm-def}{Definition}
\newmdtheoremenv{frm-lem}{Lemma}

\newtheorem{definition}{Definition}
\newtheorem{example}{Example}

\newtheorem{remark}{Remark}

\newtheorem{proof techniques}{Proof Techniques}

\newtheorem{lemma}{Lemma}
\newtheorem{corollary}{Corollary}

\newtheorem{proposition}{Proposition}

\newtheorem*{theorem*}{Theorem}
\newtheorem*{exercise*}{Exercise}
\newtheorem*{solution*}{Solution}

\newtheorem{lettertheorem}{Theorem}

\title{Antidiagonal Initial Complexes of Infinite Matrix Schubert Varieties are Cohen-Macaulay}
\author{Anna Natalie Chlopecki, Nathaniel Gallup, and Jason Meintjes} 
\date{\today}

\begin{document}
\begin{abstract}
    %We give a notion of shellability for an infinite-dimensional simplicial complex that guarantees that its (non-Noetherian) Stanley-Reisner ring is Cohen-Macaulay in the sense of ideals and is weak Bourbaki unmixed. This parallels the finite-dimensional case  developed by Stanley, Reisner, and Bj\"{o}rner (among others). We then show that initial complexes of infinite matrix Schubert varieties satisfy this new infinite shellability, giving new examples of non-Noetherian Cohen-Macaulay rings. 
    We show that, under certain constraints, the Stanley-Reisner ring of an infinite simplicial complex is Cohen-Macaulay in the sense of ideals and weak Bourbaki unmixed. 
    %This parallels the finite-dimensional case proved by Fulton. 
    We apply this result to prove the wanted claim -- that initial complexes of matrix Schubert varieties corresponding to infinite permutations in $S_{\infty}$ with respect to an antidiagonal term order are Cohen-Macaulay (in the same sense), giving rise to new examples of non-Noetherian Cohen-Macaulay rings.
\end{abstract}
\maketitle

\section{Introduction}
In 1979, Kind and Kleinschmidt proved that if a finite pure simplicial complex $\Delta$ is shellable, then its Stanley-Reisner ring $k[\Delta]$ is Cohen-Macaulay \cite{kind_schalbare_1979}. This is also implicit in Hochster's 1972 work which shows that if a finite pure simplicial complex $\Delta$ is constructible, then $k[\Delta]$ is Cohen-Macaulay \cite{hochster1972}. These results, along with those of Stanley, Reisner, and others, provide fruitful combinatorial tools to generate examples of Cohen-Macaulay rings. For example, Knutson, Miller, and Sturmfels showed that the simplicial complex associated to the initial ideal (with respect to any antidiagonal term order) of a matrix Schubert variety is shellable, implying via Kind and Kleinschmidt's result that its Stanley-Reisner ring is Cohen-Macaulay (see \cite{knutson2005grobner}, \cite{miller2005combinatorial}).

There has been interest in generalizing the Cohen-Macaulay condition to non-Noetherian rings (see \cite{glaz1992coherence}, \cite{glaz1995homological}, \cite{hamilton2007non}, and \cite{asgharzadeh2009notion}, the latter giving a nice survey). In the Noetherian case, there are many equivalent definitions of a Cohen-Macaulay ring, which do not remain so in the non-Noetherian case. Therefore, several possible definitions of Cohen-Macaulayness for non-Noetherian rings have arisen, the strongest of them being \emph{Cohen-Macaulay in the sense of ideals} and \emph{Weak Bourbaki unmixed} (see \cite{asgharzadeh2009notion}, where the somewhat complicated relationships between the various notions are explained). In \cite[Corollary 3.8]{asgharzadeh2014direct}, Asgharzadeh, Dorreh, and Tousi give a nice method for generating rings satisfying the strongest of these definitions: any flat direct limit of Noetherian Cohen-Macaulay rings is both Cohen-Macaulay in the sense of ideals and weak Bourbaki unmixed. We call such rings \emph{Cohen-Macaulay in the sense of flat direct limits}.   

It is the goal of this paper to extend the aforementioned 
%{\color{blue}shellability} 
techniques to infinite-dimensions in order to give combinatorially nice examples of non-Noetherian Cohen-Macaulay rings. We prove the following theorem, which guarantees, under certain circumstances, that the Stanley-Reisner ring of an infinite-dimensional simplicial complex is Cohen-Macaulay in the sense of flat direct limits. 

\begin{lettertheorem}\label{thm: main thm}
Let $\Delta$ be a simplicial complex so that $V(\Delta)$ is countable. Suppose there exists an increasing sequence $\Delta_1 \subseteq \Delta_2 \subseteq \ldots$ of finite full subcomplexes of $\Delta$ such that $\Delta_n$ is Cohen-Macaulay and $\bigcup_{n \in \nn} \Delta_n = \Delta$. Then, the Stanley-Reisner ring $k[\Delta]$ is a flat direct limit of the Stanley-Reisner rings $k[\Delta_n]$, and hence is Cohen-Macaulay in the sense of flat direct limits. 
\end{lettertheorem}

%Let $\Delta$ be a simplicial complex on a countable set $V$. Suppose there exists an increasing sequence $V_1 \subseteq V_2 \subseteq \ldots \subseteq V$ of finite subsets such that $\bigcup_{n \in \nn} V_n = V$ and for all $n \in \nn$ a simplicial complex $\Delta_n$ on $V_n$ such that $\Delta_n$ is Cohen-Macaulay, $\bigcup_{n \in \nn} \Delta_n = \Delta$, and $\Delta_n$ is a full subcomplex of $\Delta_{n + 1}$.

%\begin{lettertheorem}\label{thm: main thm}
%Let $\Delta$ be a countable simplicial complex. If there exists an increasing sequence $\Delta_1 \subseteq \Delta_2 \subseteq \ldots$ of finite full Cohen-Macaulay subcomplexes of $\Delta$ such that each vertex of $\Delta$ is a vertex of some $\Delta_i$, then the Stanley-Reisner ring $k[\Delta]$ is Cohen-Macaulay in the sense of flat direct limits 
%\end{lettertheorem}

There is an equivalent formulation of Theorem \ref{thm: main thm} in terms of rings, which we now state.

\begin{lettertheorem}\label{thm: main thm initial rings}
Let $X$ be a countable set, $R = k[X]$ a polynomial ring, and $J \subseteq R$ a square-free monomial ideal. Suppose there exists an increasing sequence $X_1 \subseteq X_2 \subseteq \ldots \subseteq X$ of finite subsets such that $\bigcup_{n \in \nn} X_n = X$ and for all $n \in \nn$ a square-free monomial ideal $J_n \subseteq R_n \coloneq k[X_n]$ such that $R_n / J_n$ is Cohen-Macaulay. Further, suppose that $\iota_n(J_n) \subseteq J_{n + 1}$, $\pi_n(J_{n + 1}) \subseteq J_n$, and $\bigcup_{n \in \nn} \eta_n(J_n) R = J$, where we denote by $\eta_n : R_n \to R$ and $\iota_n : R_n \to R_{n + 1}$ the inclusion maps and $\pi_n : R_{n + 1} \to R_n$ the projection map which sends $x \mapsto x$ if $x \in X_n$ and $x \mapsto 0$ otherwise. Then, $R / J$ is Cohen-Macaulay in the sense of flat direct limits. 
\end{lettertheorem}

%\begin{lettertheorem}\label{thm: main thm initial rings}
%Let $R = k[x_1, x_2, \ldots]$ be a polynomial ring with countably many variables, $J \subseteq R$ a square-free monomial ideal, and let $c: \nn \to \nn$ be an increasing function which is not eventually constant. Denote by $R_n$ the polynomial ring $k[x_1, \ldots, x_{c(n)}]$, by $\eta_n : R_n \to R$ and $\iota_n : R_n \to R_{n + 1}$ the inclusion maps, and $\pi_n : R_{n + 1} \to R_n$ the projection maps which send $x_i \mapsto x_i$ if $1 \leq i \leq c(n)$ and $x_i \mapsto 0$ otherwise. Suppose that for each $n \in \nn$, $J_n$ is a square-free monomial ideal in $R_n$ such that $\iota_n(J_n) \subseteq J_{n + 1}$, $\pi_n(J_{n + 1}) \subseteq J_n$, $\bigcup_{n \in \nn} \eta_n(J_n)R = J$, and $R_n / J_n$ is Cohen-Macaulay for all $n \in \nn$. Then $R / J$ is Cohen-Macaulay in the sense of flat direct limits. 
%\end{lettertheorem}

We apply Theorem \ref{thm: main thm initial rings} to the case of an initial ideal with respect to a term order on $R$ to obtain the following. 

\begin{lettertheorem}\label{thm: main thm rings}
Assume the same notation as in Theorem \ref{thm: main thm initial rings}. For each $n \in \nn$, let $I_n \subseteq R_n$ be an ideal such that $\iota_n(I_n) \subseteq I_{n + 1}$ and $\pi_n(I_{n + 1}) \subseteq I_n$. Suppose $<$ is a term order on $R$, $<_n$ is the restriction of $<$ to $R_n$, and $R_n / \text{in}_{<_n}(I_n)$ is Cohen-Macaulay for all $n$. Define $I\coloneq  \bigcup_{n \in \nn} \eta_n(I_n) R$. Then, $R / \text{in}_<(I)$ is Cohen-Macaulay in the sense of flat direct limits.
%{\color{blue}Question: isn't $I$ already an initial ideal?}
\end{lettertheorem}

In \cite{gallup2021well}, the second author introduced the notion of an infinite matrix Schubert variety in relation to the moduli space of full flags in a countable-dimensional vector space, whose Grothendieck ring was identified with a space of certain formal power series \cite{gallup2023grothendieck}. In this paper, we use Theorem \ref{thm: main thm rings} to prove Theorem \ref{thm: main thm msv} below.

\begin{lettertheorem}\label{thm: main thm msv}
Let $\sigma \in S_\infty$ and $I_{\sigma}$ be the Schubert determinantal ideal associated to $\sigma$. Suppose $<$ is any antidiagonal term order on $T = k[x_{i , j} \mid i , j \in \nn]$. Then, $T / \text{in}_<(I_\sigma)$ is Cohen-Macaulay in the sense of flat direct limits.
\end{lettertheorem}

The outline of the paper is as follows. In Section \ref{sec: Potentially Infinite Simplicial Complexes and Stanley-Reisner Rings}, we discuss infinite-dimensional simplicial complexes, a version of the Stanley-Reisner bijection in the infinite setting, and various properties of subcomplexes that we will need. In Section \ref{sec: Compatible Systems of Parameters}, we show that if $\Delta$ is a full subcomplex of $\Sigma$, then certain linear systems of parameters for $k[\Delta]$ can be extended to compatible systems of parameters for $k[\Sigma]$. In Section \ref{sec: Flat Homomorphisms}, we use the results from the previous two sections to show that if $\Delta$ is a full subcomplex of $\Sigma$, then the inclusion map $k[\Delta] \to k[\Sigma]$ is flat. We apply this result in Section \ref{sec: Proofs of Theorems} to chains of finite full subcomplexes of an infinite simplicial complex to prove Theorems \ref{thm: main thm} and \ref{thm: main thm initial rings}. In Section \ref{sec: Applications to Initial Ideals}, we show that by restricting monomial orders, we can extend our main theorem to initial complexes, proving Theorem \ref{thm: main thm rings}. Finally, in Section \ref{sec: Infinite Matrix Schubert Varieties}, we apply this result to infinite matrix Schubert varieties to obtain Theorem \ref{thm: main thm msv}. 

%%%%%%%%%%%%%%%%%%%%%%%%%%%%%%%%%%%%%%%%%%%%%%%%%%%%%%%%%%%%%%%%%%%%%%%%%%%%%%%%%%

\section{Potentially Infinite Simplicial Complexes and Stanley-Reisner Rings}\label{sec: Potentially Infinite Simplicial Complexes and Stanley-Reisner Rings}

A \emph{simplicial complex} on a countable set $V$ (called the \emph{vertex set}) is a collection $\Delta$ of \textbf{finite} subsets of $V$ (called \emph{faces}) with the property that if $A \in \Delta$ and $B \subseteq A$ then $B \in \Delta$. Note that if $\Delta$ is a simplicial complex on $V$ and $V \subseteq W$, then $\Delta$ is also a simplicial complex on $W$. The \emph{vertices} of a simplicial complex $\Delta$ on a set $V$ are defined to be the elements of the set $V(\Delta) = \{ v \in V \mid \{ v \} \in \Delta \}$. 

We say that $\Delta$ is a \emph{finite simplicial complex} if it is a finite set. The \emph{dimension} of $\Delta$ is $\dim \Delta = \max\{ | A | \mid A \in \Delta \} - 1$. If $V$ is finite, then clearly $\dim \Delta$ is finite as well. However, if $\Delta$ is infinite (and hence $V$ is infinite), then $\dim \Delta$ can be finite or infinite. 

A face $F$ of $\Delta$ is called a \emph{facet} if it is maximal with respect to inclusion among the set of faces. A finite-dimensional simplicial complex $\Delta$ is called \emph{pure} if all of its facets have the same size. Note that if $V$ is finite, then every face is contained in a facet, but if $\Delta$ is infinite, this is no longer true. Indeed, it may be that $\Delta$ has no facets. 

\begin{example}
    Let $V = \nn = \{ 1 , 2 , 3 , \ldots \}$, and let $\Delta$ be the set of all finite subsets of $V$. Clearly, $\Delta$ is a simplicial complex, namely the \emph{infinite-dimensional simplex}. Notice that $\Delta$ has no facets. 
\end{example}

\begin{example}\label{example:finite dim on infinite set}
    Let $V = \nn$, and let $\Delta$ be the set of all subsets of $V$ of size at most $n$, for some fixed $n \in \nn$. Then, $\Delta$ is a $(n-1)$-dimensional simplicial complex with an infinite vertex set.
\end{example}

Given a finite subset $A \subseteq V$, define $\textbf{x}^A = \prod_{v \in A} x_v \in k[x_v \mid v \in V]$. Note that this is a square-free monomial. If $\Delta$ is a simplicial complex with vertex set $V$, define $I_{\Delta, V}$ to be the ideal of $k[x_v \mid v \in V]$ generated by the set of square-free monomials $\{ \textbf{x}^A \mid A \subseteq V, A \text{ is finite}, A \notin \Delta \}$. 

When $V$ is finite, the map $\Delta \mapsto I_{\Delta, V}$ is well-known to be a bijection between the set of simplicial complexes on $V$ and the set of square-free monomial ideals in $k[x_v \mid v \in V ]$; this is due to Reisner \cite{reisner1976cohen}.
%and independently Stanley \cite{stanley1975cohen} JASON: I cannot yet verify this. 
Recall that a \emph{square-free monomial ideal} of $k[x_v \mid v \in V]$ is an ideal that can be generated by square-free monomials. It turns out that even when $V$ is infinite, this map is still a bijection. This relies on the following facts about monomial ideals in polynomial rings in countably many variables. 

There is a $\zz^{\oplus \nn}$ grading of the polynomial ring $R = k[x_1, x_2 , \ldots]$, which assigns $\deg(x_i) = e_i$ (the $i$th standard basis vector in $\zz^{\oplus \nn}$) and $\deg(a) = 0$ for all $a \in k$. Elements of $\zz^{\oplus \nn}$ will be denoted by $(a_i)_{i \in \nn}$ or just $(a_i)$, and the monomial $\prod_{i \in \nn} x_i^{a_i}$ will be denoted by $\textbf{x}^{(a_i)}$. The ideals that are homogeneous with respect to this grading are called \emph{monomial ideals}. They are exactly the ideals which can be generated by monomials (see \cite[Proposition 2]{gallup2023grothendieck}). As in the finite variable case, if $I$ is a monomial ideal and $C$ is a set of monomials which generate $I$, then a monomial $r$ of $R$ is in $I$ if and only if $r = sc$ where $c \in C$ and $s \in R$ is also a monomial. 

Denote by $<_\text{div}$ the \emph{divisibility order} on $\zz^{\oplus \nn}$, i.e. $(a_i) \leq_\text{div} (b_i)$ if and only if $a_i \leq b_i$ for all $i \in \nn$. 

\begin{lemma}\label{lem: divisibility is well-founded}
    The divisibility order is well-founded on the subset $(\zz_{\geq 0})^{\oplus \nn}$, meaning every non-empty subset has a minimal element.
\end{lemma}

\begin{proof}
    Given any sequence $(b_i) \in (\zz_{\geq 0})^{\oplus \nn}$, all but finitely many of the $b_i$ are zero. Hence, $\ell( (b_i)) \coloneq \sum_{i \in \nn} b_i \in \zz_{\geq 0}$ is finite. If $(a_i) <_\text{div} (b_i)$, then it must be that $\ell( (a_i) ) < \ell( (b_i) )$. Therefore, by applying $\ell$ to any strictly descending chain in $(\zz_{\geq 0})^{\oplus \nn}$, we obtain a strictly descending chain in $\zz_{\geq 0}$. Since the latter is well-ordered, no such chain can be infinite, and so the former is well-founded.
\end{proof}

\begin{proposition}\label{prop: infinite dickson}
    Suppose that $I \subseteq k[x_1, x_2, \ldots]$ is a monomial ideal. The set $C$ of all monomials $\textbf{x}^{(a_i)}$ with the property that $(a_i)$ is a $<_\text{div}$-minimal element of the set $\{ (b_i) \mid \textbf{x}^{(b_i)} \in I\}$ is the unique minimal (with respect to inclusion) monomial generating set of $I$. 
\end{proposition}

\begin{proof}
First, we show that $C$ generates $I$. Given a monomial $\textbf{x}^{(b_i)}$ in $I$, the set $\{ (a_i) \mid \textbf{x}^{(a_i)} \in I \text{ and } (a_i) \leq_\text{div} (b_i) \}$ is nonempty (it contains $(b_i)$), and so by Lemma \ref{lem: divisibility is well-founded}, it has a minimal element -- call it $(b'_i)$. Note that 
%$(b'_i)_{i \in \nn}$ is in fact in $C$
$\mathbf{x}^{(b'_i)}$ is in fact in $C$. Furthermore, since $(b'_i) \leq_\text{div} (b_i)$, there exists some $(c_i) \in (\zz_{\geq 0})^{\oplus \nn}$ such that $(b'_i) + (c_i) = (b_i)$, implying $\textbf{x}^{(b'_i)} \textbf{x}^{(c_i)} = \textbf{x}^{(b_i)}$. Thus, $\textbf{x}^{(b_i)}$ is in the ideal generated by $C$. Since $I$ is a monomial ideal, it is generated by the set of monomials it contains. Therefore, $I$ is generated by $C$ as well. 

Now we show that $C$ is a minimal generating set with respect to inclusion. Indeed, suppose $C'$ is a proper subset of $C$ which also generates $I$. Then, there is some $\textbf{x}^{(c_i)} \in C \smallsetminus C'$ which must be in the monomial ideal generated by $C'$, and hence must be a multiple of some monomial in $C'$, i.e. there exists some $\textbf{x}^{(c'_i)} \in C'$ such that $\textbf{x}^{(c'_i)} \textbf{x}^{(d_i)} = \textbf{x}^{(c_i)}$. But, this implies that $(c'_i) + (d_i) = (c_i)$, where $(d_i)$ is nonzero because $\textbf{x}^{(c_i)} \notin C'$ by hypothesis. Hence $(c'_i) <_\text{div} (c_i)$, which contradicts the minimality of $(c_i)$.

Finally, we show that $C$ is unique. Suppose that $D$ is another minimal monomial generating set of $I$. Then, given any $\textbf{x}^{(c_i)} \in C$, it must be that $\textbf{x}^{(c_i)}$ is a multiple of some monomial $\textbf{x}^{(d_i)} \in D$. However, as above, this implies that $(d_i) \leq_\text{div} (c_i)$, and so by minimality of $(c_i)$, it must be that $(d_i) = (c_i)$. Thus $C \subseteq D$. Since $D$ is minimal with respect to inclusion among generating sets, it must be that $C = D$, as desired. 
\end{proof}

From Proposition \ref{prop: infinite dickson}, we can easily obtain the desired bijection, the proof of which is similar to the finite case, but we include it for completeness.

\begin{proposition}\label{prop: infinite SR bijection}
    Let $V$ be any countable set. The map $\Delta \mapsto I_{\Delta, V}$ is a bijection between the set of simplicial complexes on $V$ and the set of square-free monomial ideals in $k[x_v \mid v \in V]$.
\end{proposition}

\begin{proof}
    Denote by $\mC$ the set of all simplicial complexes on $V$ and by $\mI$ the set of all square-free monomial ideals of $k[x_v \mid v \in V]$. Define the maps
    \begin{align*}
        \varphi:\mC &\to \mI \quad \text{by} \quad
        \varphi(\Delta) = \langle \mathbf{x}^A \mid A \subseteq V,\, A \text{ finite},\, A \notin \Delta \rangle, \text{ and}\\
        \psi: \mI &\to \mC \quad \text{by} \quad
        \psi(I) = \{A \subseteq V \mid A \text{ finite},\, \mathbf{x}^A \notin I\}.
    \end{align*}
    We must show that the maps above are each well-defined and inverses of one another.

    Suppose that $\Delta \in \mC$ so that $\varphi(\Delta) = \langle \mathbf{x}^A \mid A \subseteq V,\, A \text{ finite},\, A \notin \Delta \rangle$. For a prescribed generator $\mathbf{x}^A$, we must have $A=\{v_{1},\dots,v_{n}\}$ for finitely-many distinct $v_{i} \in V$, since we're given that $A \subseteq V$ and $A$ finite. Thus, $\mathbf{x}^A$ is a valid square-free monomial in $k[x_v \mid v \in V]$, and hence $\varphi(\Delta) \in \mI$. 
    %If $\Delta' \in \mC$ such that $\Delta = \Delta'$, then for every $A \notin \Delta$, we must have $A \notin \Delta'$. Thus, for every $\mathbf{x}^A \in \varphi(\Delta)$, we have $\mathbf{x}^A \in \varphi(\Delta')$. {\color{cyan}{it might be unnecessary to show this explicitly?}} The reverse inclusion follows similarly. Hence, $\varphi$ is well-defined.

    Now suppose that $I \in \mI$ so that $\psi(I) = \{A \subseteq V \mid A \text{ finite},\, \mathbf{x}^A \notin I\}$. Let $F \in \psi(I)$ be an arbitrary face and, in search of a contradiction, suppose that for some $G \subseteq F$, we have that $G \notin \psi(I)$. Since $F \in \psi(I)$, we have that $F$ is a finite subset of $V$ and that $\mathbf{x}^F \notin I$. As $F$ is finite, $G \subseteq F$ must be finite too. So, 
    the only condition preventing $G$ from being in $\psi(I)$ must be that $\mathbf{x}^G \in I$. But then, by ideal closure, $\mathbf{x}^G \cdot \mathbf{x}^{F \smallsetminus G} = \mathbf{x}^F \in I$, a contradiction. We conclude that $\psi(I) \in \mC$ is a valid simplicial complex. 
    %If $I = I'$ for some other $I' \in \mI$ then, by Proposition \ref{prop: infinite dickson}, the two ideals have the same minimal generating set and thus define the same complex. The map $\psi$ is well-defined. {\color{cyan}{it also might be unnecessary to show this explicitly?}}

    We now verify that $\psi \circ \varphi = \id_\mC$ and $\varphi \circ \psi = \id_\mI$. For the first composition, let $\Delta \in \mC$; we'll show that $\psi(\varphi(\Delta)) = \Delta$. Let $F \in \Delta$ be an arbitrary face. Since $F \in \Delta$, we have that $F \subseteq V$, $F$ is finite, and $\mathbf{x}^F \notin \varphi(\Delta)$, implying $F \in \psi(\varphi(\Delta)) \coloneq \{A \subseteq V \mid A \text{ finite},\, \mathbf{x}^A \notin \varphi(\Delta)\}$ so that $\Delta \subseteq \psi(\varphi(\Delta))$. For the reverse inclusion, suppose $G \in \psi(\varphi(\Delta))$ is an arbitrary face so that $G \subseteq V$, $G$ finite, and $\mathbf{x}^G \notin \varphi(\Delta)$. Then, since $G$ is a finite subset of $V$ and the monomial $\mathbf{x}^G$ is not in $\varphi(\Delta)$, we must have that $G \in \Delta$. So, $\psi(\varphi(\Delta)) \subseteq \Delta$, and hence $\Delta = \psi(\varphi(\Delta))$.

    For the second composition, let $I \in \mI$; we'll show that $\varphi(\psi(I)) = I.$ 
    Suppose that $\mathbf{x}^{A'} \in I$ is a minimal generator for $I$ and, in search of a contradiction, that $\mathbf{x}^{A'}$ is not a minimal generator for $\varphi(\psi(I))$. Since $\mathbf{x}^{A'} \in I$, $A'$ is a finite subset of $V$. Note that $A' \notin \psi(I)$. Hence, $\mathbf{x}^{A'} \in \varphi(\psi(I)).$ As we've now shown that $I \subseteq \varphi(\psi(I))$, and we're assuming $\mathbf{x}^{A'}$ is not a minimal generator for $\varphi(\psi(I))$, we must have that $\mathbf{x}^{A'} = \mathbf{x}^B \cdot \mathbf{x}^C$ for some minimal generator $\mathbf{x}^B$ of $\varphi(\psi(I))$ and some nonunit $\mathbf{x}^C$, contradicting the minimality of $\mathbf{x}^{A'}$.
    A similar argument shows that $\varphi(\psi(I)) \subseteq I$ and that every minimal generator of $\varphi(\psi(I))$ is a minimal generator of $I$. As the two ideals have the same minimal generating set, again by Proposition \ref{prop: infinite dickson}, $\varphi(\psi(I)) = I$.
\end{proof}

%{\color{orange}{I think we should leave the above proof in! I think it should be written down somewhere in the literature and this is the perfect place. However, I think we can comment out some of the details to make it shorter (so if the reader wants they can follow it and reprove it themselves).}} {\color{cyan}{Agreed - I like it here!}}

\begin{remark}
    There are several different possible notions of a simplicial complex on an infinite set. 
    %{\color{orange}{Modify words to make it clear that remark is with regards to finiteness of subsets.}} 
    For example, one could define a simplicial complex on $V$ to be \textbf{any} (possibly infinite) subset of the power set of $V$ which is closed under taking subsets. However, with this definition, it is not true that $\Delta \mapsto I_{\Delta, V}$ is a bijection. For example, if $V = \nn$ and $\Delta$ is the set of all proper subsets of $V$ (i.e. not including $V$ itself), while $\Delta'$ is the entire power set, then $I_{\Delta, V} = I_{\Delta', V} = 0$. 
\end{remark}

If $\Delta$ is a simplicial complex on a vertex set $V$, then its \emph{Stanley-Reisner ring} (or \emph{face ring}) is defined to be $k[\Delta] \coloneq k[x_v \mid v \in V]/I_{\Delta, V}$. Note that while this appears to depend on $V$, if $\Delta$ is also a simplicial complex on $W$, then there is a canonical isomorphism $k[x_v \mid v \in V]/I_{\Delta, V} \cong k[x_w \mid w \in W]/I_{\Delta, W}$, since any element of $V$ which is not a vertex of $\Delta$ is killed in the quotient (and similarly for $W$).

%\begin{example}
%    If $P=(X,\le)$ is a poset, $I_{\Delta(P),X}$ is generated by the square-free monomials of the form $x_p x_q$ where $p$ and $q$ are not related in $P$.
%\end{example}

%\begin{definition}
%    If $\Delta$ and $\Sigma$ are simplicial complexes, we say that $\Delta$ is a \emph{subcomplex} of $\Sigma$ if $\Delta \subseteq \Sigma$. We furthermore say that $\Delta$ is a \emph{full subcomplex} of $\Sigma$ if whenever $A \in \Sigma$ and for all $w \in A$ we have that $w \in V(\Delta)$, then $A \in \Delta$. 
    %{\color{orange}{(I changed the definition again because it was bothering me - I think it is better now.)}}
%\end{definition}

\begin{definition}
    If $\Delta$ and $\Sigma$ are simplicial complexes, we say that $\Delta$ is a \emph{subcomplex} of $\Sigma$ if $\Delta \subseteq \Sigma$. We say that $\Delta$ is a \emph{full complex} of $\Sigma$ if whenever $A \in \Sigma$ and for all $w \in A$ we have that $w \in V(\Delta)$, then $A \in \Delta$. Finally, if $\Delta$ is a subcomplex of $\Sigma$ and a full complex of $\Sigma$, then we say that $\Delta$ is a \emph{full subcomplex} of $\Sigma$.  
    %{\color{orange}{(I changed the definition again because it was bothering me - I think it is better now.)}}
\end{definition}

\begin{example}
    Let $V = [n] \coloneq \{1, 2, ,\dots, n\}$ for some fixed $n \in \nn$ where $n>2$, and $W = \nn$. Let $\Delta$ be the set of all subsets of $V$ excluding those subsets that contain $\{1,2\}$, and let $\Sigma$ be the set of all finite subsets of $W$ excluding those subsets that contain $\{1,2\}$ or $\{n+1,n+2\}$. Then, $\Delta$ is a finite simplicial complex with facets $\{1,3,4,\dots,n\}$ and $\{2,3,4,\dots,n\}$, while $\Sigma$ is an infinite simplicial complex with no facets. We have that $k[\Delta] = k[x_1,\dots,x_n]/\langle x_1x_2\rangle$
    and
    $k[\Sigma]=k[x_1,x_2,\dots]/\langle x_1x_2, x_{n+1}x_{n+2} \rangle.$
    Note that $\Delta$ is a full subcomplex of $\Sigma$.
\end{example}

Suppose $V \subseteq W$. If $\Delta$ is a simplicial complex on $V$, $\Sigma$ is a simplicial complex on $W$, and $\Delta$ is a subcomplex of $\Sigma$, then we have the natural inclusion $\iota: k[x_v \mid v \in V] \to k[x_w \mid w \in W]$ which sends $x_v \mapsto x_v$, and the natural projection $\pi : k[x_w \mid w \in W] \to k[x_v \mid v \in V]$ which sends $x_w$ to $x_w$ if $w \in V$ and to $0$ otherwise. Note that $\pi \circ \iota = \id_{k[x_v \mid v \in V]}$. The following proposition shows that these maps descend to maps on the face rings. 

\begin{proposition} \label{prop: maps up and down full complexes}
    Let $\Delta$ be a simplicial complex on $V$, $\Sigma$ be a simplicial complex on $W$, and assume that $V \subseteq W$. 
    
    \begin{enumerate}
        \item The following are equivalent. 
            \begin{enumerate}
                \item $\Delta$ is a subcomplex of $\Sigma$.
                \item $\pi(I_{\Sigma , W}) \subseteq I_{\Delta , V}$. 
                \item $\pi$ induces a map $\overline{\pi} : k[\Sigma] \to k[\Delta]$.
            \end{enumerate}
        \item The following are equivalent.
            \begin{enumerate}
                \item $\Delta$ is a full complex of $\Sigma$.
                \item $\iota(I_{\Delta , V}) \subseteq I_{\Sigma , W}$. 
                \item $\iota$ induces a map $\overline{\iota} : k[\Delta] \to k[\Sigma]$. 
            \end{enumerate}
    \end{enumerate}

If both sets of conditions hold, then $\overline{\pi} \circ \overline{\iota} = \id_{k[\Delta]}$.
\end{proposition}

\begin{proof}
In both cases, the equivalence of (b) and (c) is trivial. In case (1), to see the equivalence of (a) and (b), suppose that $\Delta$ is a subcomplex of $\Sigma$ and that $A\subseteq W$ is a finite set which is not a face of $\Sigma$. On the one hand, if $A \not\subseteq V$, then there exists $w \in A \smallsetminus V$, and by definition $\pi ( x_w ) = 0$, and so we must have that $\pi ( \textbf{x}^A ) = 0$. On the other hand, if $A \subseteq V$, then $A$ cannot be a face of $\Delta$, as $\Delta$ is a subcomplex of $\Sigma$ and $A$ is not a face of $\Sigma$. Thus, $\pi ( \textbf{x}^A ) =\textbf{x}^A \in I_{\Delta , V}$. Conversely, suppose that $\pi(I_{\Sigma , W}) \subseteq I_{\Delta , V}$. Let $A \in \Delta$. If $A \notin \Sigma$, then $\textbf{x}^A \in I_{\Sigma , W}$, and so $\textbf{x}^A = \pi(\textbf{x}^A) \in I_{\Delta , V}$. This implies that $\textbf{x}^A$ is a multiple of $\textbf{x}^B$ for some $B$ which is not a face of $\Delta$. Therefore, $B \subseteq A$, implying that $B$ is a face of $\Delta$, a contradiction.

To see the equivalence of (a) and (b) in case (2), suppose $\Delta$ is a full complex of $\Sigma$ and that $A \subseteq V$ is not a face of $\Delta$. Then, $A$ cannot be a face of $\Sigma$, since otherwise it would have to be a face of $\Delta$ by the ``full subcomplex'' hypothesis. Thus, $\iota( \textbf{x}^A ) \in I_{\Sigma,W}$, implying $\iota( I_{\Delta , V} ) \subseteq I_{\Sigma , W}$.  Conversely, suppose that $\iota$ induces a map $\overline{\iota} : k[\Delta] \to k[\Sigma]$. So, we must have that $\iota( I_{\Delta , V} ) \subseteq I_{\Sigma,W}$. Let $A \in \Sigma$ be such that $w \in V(\Delta)$ for all $w \in A$. If $A \notin \Delta$, then $\textbf{x}^A \in I_{\Delta , V}$. Hence, $\textbf{x}^A = \iota(\textbf{x}^A) \in I_{\Sigma , W}$. So, $\textbf{x}^A$ is a multiple of $\textbf{x}^B$ for some $B$ which is not a face of $\Sigma$. But, again, this means that $B \subseteq A$, and so $B \in \Sigma$, a contradiction.

If both cases hold, then since $\pi \circ \iota = \id_{k[x_v \mid v \in V]}$, it follows trivially that $\overline{\pi} \circ \overline{\iota} = \id_{k[\Delta]}$. 
\end{proof}

%%%%%%%%%%%%%%%%%%%%%%%%%%%%%%%%%%%%%%%%%%%%%%%%%%%%%%%%%%%%%%%%%%%%%%%%%%%%%%%%%%

\section{Compatible Systems of Parameters}\label{sec: Compatible Systems of Parameters}

In this section, let $k$ be an algebraically closed field. Recall that for a positively graded affine $k$-algebra $R$ with $\dim R = d$, a set of homogeneous elements $\theta_1,\dots, \theta_d$ is called a \emph{homogeneous system of parameters} if and only if $R$ is an integral extension of $k[\theta_1,\dots,\theta_d]$ if and only if $R$ is a finite $k[\theta_1,\dots,\theta_d]$-module. A standard result (cf. \cite[1.5.17]{bruns1998cohen}) says that such a set always exists and is, moreover, algebraically independent over $k$. Furthermore, in the case that $k$ is infinite, the $\theta_i$ can be taken to be linear.

In \cite{stanley1979balanced}, Stanley gives the the following matrix criterion for testing whether a collection of linear forms in a Stanley-Reisner ring forms a system of parameters (SOP). 

%{\color{orange}{Define SOP using graded k-algebra case - non-local case, Brunz-Herzog CM rings.}} 

\begin{lemma}\cite[Remark on page 150]{stanley1979balanced}\label{lem: stanley's criterion}
    Suppose that $k$ is a field, $\Delta$ is a $(d - 1)$-dimensional pure simplicial complex on a vertex set $\{ 1 , \ldots, n \}$, and that for $1 \leq i \leq d$, $\theta_i = \sum_{j = 1}^n a_{i , j} x_j \in k[\Delta]$ is a linear form. Then $\theta_1, \ldots, \theta_d$ is a system of parameters for $k[\Delta]$ if and only if for all facets $F$ of $\Delta$ the $d \times d$ minor of the $d \times n$ matrix $[a_{i , j}]$ with columns indexed by vertices of $F$ is non-singular.
(Equivalently, for all faces $F$ of $\Delta$ the associated $d \times |F|$ submatrix of the $d \times n$ matrix $[a_{i , j}]$ with columns indexed by vertices of $F$ has rank $|F|$.)
    
\end{lemma}

%{\color{orange}{Delete comment below lemma and make lemma about faces (not facets). We still want to use "good" though.}}

If $\theta_1, \ldots, \theta_d$ is a set of linear forms in some Stanley-Reisner ring $k[\Delta]$ which is pure of dimension $d - 1$ with $\theta_i = \sum_{j = 1}^n a_{i , j} x_j \in k[\Delta]$, then we call the $d \times n$ matrix $[a_{i , j}]$ the \emph{matrix of the set of linear forms}. If this matrix has the property that \textbf{all} minors with columns indexed by a face of $\Delta$ do not vanish, then we call the set of linear forms \emph{good}. 
%Note that the ``good'' property is stronger than is needed to guarantee that the set of linear forms is a system of parameters by Stanley's criterion, but it will be helpful for us.  

%{\color{blue}{Make $\delta$ to $\theta$ in prop. 4 and cor. 1.}}

\begin{proposition} \label{prop: generic compatible systems of parameters}
    Let $\Delta$ be a finite pure simplicial complex on the vertex set $\{ 1 , \ldots, m \}$ of dimension $d - 1$ and $\Sigma$ be a finite pure simplicial complex on the vertex set $\{ 1 , \ldots, n \}$ of dimension $e - 1$ such that $\Delta$ is a full subcomplex of $\Sigma$ (hence in particular $m \leq n$ and $d \leq e$). Suppose that $\theta_1, \ldots, \theta_d$ is a good linear SOP for $k[\Delta]$. Then there exists a good linear SOP $\delta_1, \ldots, \delta_e$ for $k[\Sigma]$ such that $\overline{\pi}(\delta_i) = \theta_i$ for $1 \leq i \leq d$, where $\overline{\pi}: k[\Sigma] \to k[\Delta]$ is the natural projection map of Proposition \ref{prop: maps up and down full complexes}. 
\end{proposition}

\begin{proof}
Consider the polynomial rings $R = k[y_{i , j} \mid 1 \leq i \leq e , 1 \leq j \leq n]$ and $S = k[z_{i , j} \mid d < i \leq e \text{ or } m < j \leq n]$. If $[a_{i , j}]$ denotes the $d\times m$ matrix of the set $\{ \theta_1, \ldots, \theta_d\}$ of linear forms, there is a surjective ring homomorphism $\varphi: R \to S$ which is evaluation at $(a_{i , j})$, i.e. which sends $y_{i , j} \mapsto a_{i , j}$ for $1 \leq i \leq d$ and $1 \leq j \leq m$ and $y_{i , j} \mapsto z_{i , j}$ otherwise. 

If $F = \{ j_1 , \ldots , j_b \}$ is a face of $\Sigma$ (with $j_1 < \ldots < j_b$), and $p \in R$ denotes a minor (of the matix of variables $[y_{i,j}]_{1\leq i\leq e, 1\leq j\leq n}$) whose columns are indexed by the elements of $F$ and rows are indexed by some set $\{ i_1, \ldots, i_b \}$ (with $i_1 < \ldots < i_b$), we wish to show that $\varphi(p)$ is a nonzero polynomial in $S$. Let $1 \leq a \leq b$ be maximal such that $i_a \leq d$ and $j_a \leq m$. Then, notice that the coefficient on the monomial $z_{i_{a + 1}, j_{a + 1}} \ldots z_{i_{b}, j_{b}}$ in the minor $\varphi(p) \in S$ is $\det(a_{\{i_1, \ldots, i_a\}, \{j_1, \ldots, j_a\}})$. Since $\{ j_1, \ldots, j_b \}$ is a face of $\Sigma$, $\{ j_1, \ldots, j_a \}$ is also a face of $\Sigma$. Because $\Delta$ is a full subcomplex of $\Sigma$, $\{ j_1, \ldots, j_a \}$ is a face of $\Delta$ as well. We know that $\theta_1, \ldots, \theta_d$ is a good system of parameters, implying $\det(a_{\{i_1, \ldots, i_a\}, \{j_1, \ldots, j_a\}}) \neq 0$. So, the polynomial $\varphi(p)$ has a monomial term with a nonzero coefficient, and is therefore nonzero. 

Since there are finitely many faces of $\Sigma$, there are finitely many minors in $R$ with columns indexed by these faces. We have just shown that the images of these polynomials under $\varphi$ are all nonzero, hence the finitely many non-vanishing sets of these images are non-empty Zariski open subsets of the irreducible affine space $\text{MaxSpec}(S)$, and must therefore have non-trivial intersection. 
Choose any point $(b_{i , j} \mid d < i \leq e \text{ or } m < j \leq n)$ in this intersection, and for $1 \leq i \leq d$ and $1 \leq j \leq m$, define $b_{i , j} = a_{i , j}$. Then, by construction, the $e\times n$ matrix $[b_{i , j}]$ has the property that all minors with columns indexed by faces of $\Sigma$ do not vanish. Letting $\delta_i = \sum_{j = 1}^n b_{i , j} x_j$ therefore gives the desired good linear system of parameters for $k[\Sigma]$. Furthermore, since $\overline{\pi}(x_j) = 0$ if $m < j \leq n$ and $b_{i , j} = a_{i , j}$ if $1 \leq i \leq d$ and $1 \leq j \leq m$, it follows that $\overline{\pi}(\delta_i) = \theta_i$ if $1 \leq i \leq d$, as desired. 
%Hence we have a $d \times m$ matrix $(a_{i , j})$ whose minors with columns indexed by a face of $\Delta$ do not vanish, and we wish to find an $e \times n$ matrix $(b_{i , j})$ whose minors with columns indexed by a face of $\Sigma$ do not vanish, and so that $b_{i , j} = a_{i , j}$ for $1 \leq i \leq d$ and $1 \leq j \leq m$.
\end{proof}

\begin{corollary}\label{cor: countably many compatible linear sops}
    Let $\Delta$ be a (not necessarily finite) simplicial complex on a vertex set $V$ and $\Delta_1 \subseteq \Delta_2 \subseteq \ldots$ be an increasing sequence of finite full subcomplexes of $\Delta$ which are pure and have dimension $d_i-1$, 
    %{\color{orange}{(Should the dimension be $d_i-1$? Or we have to alter things later in the proof?)}} 
    respectively, so that in particular, $d_i \leq d_{i + 1}$ for all $i \in \nn$. Then, there exist compatible linear systems of parameters for the sequence, in the sense that for each $i \in \nn$, there exists a linear system of parameters $\theta_{i , 1}, \ldots \theta_{i , d_i}$ such that $\overline{\pi}_{i + 1}(\theta_{i + 1 , j}) = \theta_{i, j}$ for $1 \leq j \leq d_i$ where $\overline{\pi}_{i+1}:k[\Delta_{i+1}]\to k[\Delta_i]$. 
\end{corollary}

\begin{proof}
    First of all, let us label the set $\{ v \in V \mid \exists i \in \nn ,\, v \in V(\Delta_i) \}$ (which is a union of countable sets and therefore must be countable) by natural numbers so that for each $i$, the vertices involved in $\Delta_i$ are $\{ 1 , \ldots , n_i \}$. Note that, in particular, $n_i \leq n_{i + 1}$ for all $i$. 
    To prove the desired result, we will show that each system of parameters can be chosen to be good and to satisfy the required property under the (induced) natural projection maps, namely $\overline{\pi}_i$. 
    %{\color{orange}{(Do we want to use the language of "stronger fact," now that we are modifying the previous comment post lemma 1?)}} 
    We induct on $i$. For the base case, consider the ideal $I$ in the polynomial ring $k[y_{i , j} \mid 1 \leq i \leq d_1, \, 1 \leq j \leq n_1]$ generated by the minors whose columns are indexed by faces of $\Delta_1$. Then, $I$ is contained in the maximal ideal $\langle y_{i , j} \rangle$, and since $k$ is algebraically closed, its non-vanishing set $D(I) \subseteq \mathbb{A}^{d_1 \times n_1}_k$ is a non-empty Zariski open set. Choose any point $(a_{\alpha , \beta}) \in D(I)$ and let 
    %{\color{orange}{(Notation should be $\theta_{1,i}$?)}} $\theta_i = \sum_{j = 1}^{n_1} a_{i , j} x_j \in k[\Delta_i]$ for $1 \leq i \leq d_1$.
    $\theta_{1,\alpha} = \sum_{\beta = 1}^{n_1} a_{\alpha , \beta} x_\beta \in k[\Delta_1]$ for $1 \leq \alpha \leq d_1$. 
    Then, by Stanley's criterion (Lemma \ref{lem: stanley's criterion}), 
    %$\theta_1, \ldots, \theta_{d_1}$ {\color{orange}{(Notation should be $\theta_{1,i}$?)}} 
    $\theta_{1,1}, \ldots, \theta_{1,d_1}$ 
    %{\color{orange}{(Notation should be $\theta_{1,i}$?)}} 
    form a good linear system of parameters for $k[\Delta_1]$. For the inductive step, suppose that for each $1 \leq \ell \leq i$ we have chosen a good linear system of parameters $\theta_{\ell , 1}, \ldots, \theta_{\ell , d_\ell} \in k[\Delta_\ell]$ where for $1 \leq \ell < i$, we have that $\overline{\pi}_{\ell + 1}(\theta_{\ell + 1 , j}) = \theta_{\ell , j}$ for $1 \leq j \leq d_\ell$. By Proposition \ref{prop: generic compatible systems of parameters}, there exists a good linear system of parameters $\theta_{i + 1 , 1}, \ldots, \theta_{i + 1 , d_{i + 1}}$ of $k[\Delta_{i + 1}]$ such that $\overline{\pi}_{i + 1}(\theta_{i + 1 , j}) = \theta_{i , j}$ for $1 \leq j \leq d_i$.
\end{proof}

%%%%%%%%%%%%%%%%%%%%%%%%%%%%%%%%%%%%%%%%%%%%%%%%%%%%%%%%%%%%%%%%%%%%%%%%%%%%%%%%%%

\section{Flat Homomorphisms}\label{sec: Flat Homomorphisms}

We recall some terminology that will be needed in this section. Suppose $f: A \to B$ is any ring homomorphism. We denote by $f_!$ the \emph{extension of scalars} functor from the category of $A$-modules to the category of $B$-modules that sends $M \mapsto M \otimes_A B$, and by $f^*$ the \emph{restriction of scalars} functor from the category of $B$-modules to the category of $A$-modules that sends $N$ to the $A$-module with underlying abelian group $N$ and $A$-multiplication given by $a \cdot n = f(a) \, n$. We say that $f$ is \emph{flat} if $f_!$ sends short exact sequences to short exact sequences, and $f$ is \emph{faithfully flat} if a complex of $A$-modules $0 \to L \to M \to N \to 0$ is exact if and only if the complex of $B$-modules $0 \to f_! L \to f_! M \to f_! N \to 0$ is exact. 

We now collect several technical results about these functors into a lemma that will be needed in the proof of the main result of this section. 

\begin{lemma}\label{lem: pushforward pullback properties}
Let $f: A \to B$, $g : B \to C$ be ring homomorphisms, and denote the identity homomorphism on $A$ by $\id_A$.   
\begin{enumerate}
    \item The functor $g_! \circ f_!$ from the category of $A$-modules to the category of $C$-modules is naturally isomorphic to the functor $(g \circ f)_!$. 
    \item The functor $(\id_A)_!$ is naturally isomorphic to the identity functor on the category of $A$-modules.
    \item If $f^* B$ is a nonzero free $A$-module, then $f$ is faithfully flat. In particular, for any set of variables $X$, the inclusion map $A \mapsto A[X]$ is faithfully flat.
    \item If $f^* B$ is a nonzero free $A$-module with $A$-basis $\{ b_i \mid i \in I\}$ and $g^* C$ is a nonzero free $B$-module with basis $\{ c_j \mid j \in J \}$, then $(g \circ f)^* C$ is a nonzero free $A$-module with basis $\{ g(b_i) c_j \mid (i , j) \in I \times J \}$.
    \item If $f^* B$ is a nonzero free $A$-module with $A$-basis $\{ b_i \mid i \in I\}$, then the functor $f_! \circ f^*$ is naturally isomorphic to the functor $(-)^{\oplus I}$, which sends a $B$-module $N$ to the $B$-module $N^{\oplus I}$.
    \item If $K$ is any set, the functor $f_! \circ (-)^{\oplus K}$ (where here $(-)^{\oplus K}$ is an endofunctor of the category of $A$-modules) is naturally isomorphic to the functor $(-)^{\oplus K} \circ f_!$ (where here $(-)^{\oplus K}$ is an endofunctor of the category of $B$-modules). 
    \item If $F$ and $G$ are naturally isomorphic additive functors between abelian categories, then $F$ is exact if and only if $G$ is exact. 
\end{enumerate}
\end{lemma}

\begin{proof}
(1) follows because there is a natural transformation $g_! \circ f_! \to (g \circ f)_!$ which, for any $A$-module $M$, is the $C$-module isomorphism $(M \otimes_A B ) \otimes_B C \to M \otimes_A C$ defined by $m \otimes b \otimes c \mapsto m \otimes g(b) c$. (2) follows because there is a natural transformation $(\id_A)_! \to \id_{A-\text{mod}}$ which, for any $A$-module $M$, is the $A$-module isomorphism $M \otimes_A A \mapsto M$ defined by $m \otimes a \mapsto am$. (3) follows because free modules are flat; see Example 3.1 \cite{MR266911}. To see why (4) is true, we show spanning and independence. Given any element $c \in C$, we can write $c = \sum_j g(\beta_j) c_j$ for some $\beta_j \in B$. Additionally, for each $j$, we can write $\beta_j = \sum_i f(\alpha_{i , j}) b_i$ for some $\alpha_{i , j} \in A$. Thus, $c = \sum_{i , j} g(f(\alpha_{i , j})) g(b_i) c_j$, completing the proof of spanning. Now, to show independence, suppose that $0 = \sum_{i , j} g(f(\alpha_{i , j})) g(b_i) c_j$ for some $\alpha_{i , j} \in A$. Then, we obtain $0 = \sum_j g( \sum_i f(\alpha_{i , j}) b_i ) c_j$, and because $\{ c_j \mid j \in J \}$ is independent, we must have $g( \sum_i f(\alpha_{i , j}) b_i ) = 0$ for all $j$. Furthermore, because $g^*C$ is a nonzero free $B$-module, it must be that $g$ is injective (otherwise, $g(b) c_j = 0$ would produce a non-trivial dependence relation for any nonzero $b \in \ker g$), and so we have that $\sum_i f(\alpha_{i , j}) b_i = 0$ for all $j$. Independence of $\{ b_i \mid i \in I \}$ implies that $f(\alpha_{i , j}) = 0$ for all choices of $i$ and $j$. Since $f^*B$ is a nonzero free $A$-module, $f$ must be injective. So, we obtain $\alpha_{i , j} = 0$ for all choices of $i$ and $j$, as desired. (5) follows because there is a natural isomorphism $(-)^{\oplus I} \to f_! \circ f^*$ which on any $B$-module $N$ is the map defined by $N^{\oplus I} \to f^* N \otimes_A B$ where $(n_i)_{i \in I} \mapsto n_i \otimes b_i$. (6) follows because the tensor product commutes with the direct sum. (7) is clear. 
\end{proof}

Recall that for a positively graded $k[x_1,\dots,x_n]$-module $M$, a sequence of homogeneous elements $y_1,\dots,y_r$ in $k[x_1,\dots,x_n]$ is called a \emph{regular sequence (of length $r$) on $M$} if 
\begin{enumerate}[label={(\roman*)}]
    \item $M/\langle y_1,\dots,y_r \rangle M \neq 0$; and
    \item For each $i = 1,\dots,r$, $y_i$ is a non-zerodivisor on $M/ \langle y_1,\dots,y_{i-1} \rangle M$.
\end{enumerate}
For our purposes, the following set of equivalences suffices as a characterization of the \textit{Cohen-Macaulay} property.

\begin{proposition}\cite[13.37]{miller2005combinatorial}
    Let $M$ be a finitely generated module of dimension $d$ over a positively graded Noetherian polynomial ring. The following are equivalent.
    \begin{enumerate}
        \item $M$ is Cohen-Macaulay;
        \item There exists a regular sequence of length $d$ on $M$;
        \item Every (homogeneous) system of parameters for $M$ is a regular sequence on $M$.
        \item $M$ is a free module over $k[y_1,\dots,y_d]$ for some (equivalently, every) homogeneous system of parameters $y_1,\dots,y_d$ for $M$.
    \end{enumerate}
\end{proposition}
We say a commutative Noetherian ring is Cohen-Macaulay if it is a Cohen-Macaulay module over itself and that a simplicial complex $\Delta$ is Cohen-Macaulay if its corresponding face ring $k[\Delta]$ is. 

\begin{remark}\label{rmk: CM is pure}
    It is a fact that Cohen-Macaulay complexes are automatically pure. See, e.g., the discussion in \cite{stanley1979balanced} after Theorem 1.2. 
\end{remark}

\begin{proposition}\label{prop: inclusion maps are flat}
    Suppose that $R$ and $S$ are finitely generated Cohen-Macaulay $k$-algebras with systems of parameters $y_1, \ldots, y_m$ and $z_1, \ldots, z_n$, respectively. Furthermore, suppose that $\iota: R \to S$ and $\pi: S \to R$ are $k$-algebra homomorphisms such that $\pi \circ \iota = \id_R$ (which, in particular, implies $m \leq n$) and $\pi(z_i) = y_i$ if $1 \leq i \leq m$. 
    %{\color{orange}{(Do you mean $\pi(z_i)=y_i$?)}} 
    Then, $\iota$ is flat.
\end{proposition}

\begin{proof}
Because $y_1, \ldots, y_m$ is a system of parameters for $R$ and $R$ is Cohen-Macaulay, the map $\alpha: R' \coloneq k[y_1, \ldots, y_m] \to R$ which sends $y_i \mapsto y_i$ makes $R$ into a free $R'$-module, say with basis $\{ r_i \mid i \in I \}$, and $\beta: S': = k[z_1, \ldots, z_n] \to S$ sending $z_i \mapsto z_i$ makes $S$ into a free $S'$-module, say with basis $\{ s_j \mid j \in J\}$. Let $\varphi: R' \to S'$ denote the map which for $1 \leq i \leq m$ sends $y_i$ to $z_i$. Then, $\varphi$ makes $S'$ into a free $R'$-module with a basis of monomials in the variables $z_{m + 1}, \ldots, z_n$, indexed by the set $T$. By parts (3) and (4) of Lemma \ref{lem: pushforward pullback properties}, $\alpha$ and $\beta \circ \varphi$ are faithfully flat. We have the following diagram. 
\begin{center}
\begin{tikzcd}
R \arrow[bend left = 20, rr, "\iota"] & & S \arrow[bend left = 10, ll, "\pi"] 
\\ R' = k[y_1, \ldots, y_m]  \arrow[u, "\alpha"] \arrow[rr, "\varphi" '] && S' = k[z_1, \ldots, z_n] \arrow[u, "\beta" ']  &
\end{tikzcd}
\end{center}

Note that this diagram does not commute under the hypotheses of the proposition, since we have not required that $\iota(y_i) = z_i$. However, it does commute if we remove the map $\iota$, i.e. we have that $\pi \circ \beta \circ \varphi = \alpha$. 

Suppose $0 \to L \to M \to N \to 0$ is an exact sequence of $R$-modules. We wish to show that $0 \to \iota_! L \to \iota_! M \to \iota_! N \to 0$ is an exact sequence of $S$-modules. However, since $\alpha$ is faithfully flat, this is equivalent to the statement that the following complex is exact.
\begin{equation}\label{eq: exact candidate 1}
    0 \to \alpha_! (\beta \circ \varphi)^* \iota_! L \to \alpha_! (\beta \circ \varphi)^* \iota_! M \to \alpha_! (\beta \circ \varphi)^* \iota_! N \to 0
\end{equation} 
But, $\alpha = \pi \circ \beta \circ \varphi$. Hence, by part (1) of Lemma \ref{lem: pushforward pullback properties}, the functor $\alpha_!$ is naturally isomorphic to the functor $\pi_! \circ (\beta \circ \varphi)_!$, implying that $\alpha_! \circ (\beta \circ \varphi)^*$ is naturally isomorphic to $\pi_! \circ (\beta \circ \varphi)_! \circ (\beta \circ \varphi)^*$. By part (5) of Lemma \ref{lem: pushforward pullback properties}, $(\beta \circ \varphi)_! \circ (\beta \circ \varphi)^*$ is naturally isomorphic to the functor $(-)^{\oplus K}$ where $K \coloneq T \times J$. Furthermore, by part (6) of Lemma \ref{lem: pushforward pullback properties}, we have that $\pi_! \circ (-)^{\oplus K} \circ \iota_!$ is naturally isomorphic to $\pi_! \circ \iota_! \circ (-)^{\oplus K}$, which, by part (1) of Lemma \ref{lem: pushforward pullback properties}, is naturally isomorphic to $(\pi \circ \iota)_! \circ (-)^{\oplus K}$. This, in turn, is naturally isomorphic to $(\id_R)_! \circ (-)^{\oplus K}$, which, by part (2) of Lemma \ref{lem: pushforward pullback properties}, is naturally isomorphic to $(-)^{\oplus K}$. All in all, the functor $\alpha_! (\beta \circ \varphi)^* \iota_!$ is naturally isomorphic to the functor $(-)^{\oplus K}$. So, by part (7) of Lemma \ref{lem: pushforward pullback properties}, the complex in Equation \ref{eq: exact candidate 1} is exact if and only if the following complex is exact. \begin{equation}\label{eq: exact candidate 2}
    0 \to L^{\oplus K} \to M^{\oplus K} \to N^{\oplus K} \to 0
\end{equation}
Notice that the complex in Equation \ref{eq: exact candidate 2} is exact since $0 \to L \to M \to N \to 0$ is by hypothesis.
\end{proof}

%%%%%%%%%%%%%%%%%%%%%%%%%%%%%%%%%%%%%%%%%%%%%%%%%%%%%%%%%%%%%%%%%%%%%%%%%%%%%%%%%%

\section{Proofs of Theorems \ref{thm: main thm} and \ref{thm: main thm initial rings}}\label{sec: Proofs of Theorems}

To fix notation, we begin with a brief review of the notion of a direct limit. A \emph{directed set} is a poset $(P , \leq)$ with the property that for all $p , q \in P$, there exists some $u \in P$ with $p , q \leq u$. A \emph{direct system} of objects of a category $\mathcal{C}$ is a family $\{ A_p , f_{p , q} \}$ where $\{ A_p \}_{p \in P}$ is a family of objects of $\mathcal{C}$ and $\{ f_{p , q} : A_p \to A_q \mid p \leq q \}_{p , q \in P}$ is a family of morphisms of $\mathcal{C}$ such that $f_{p , p} = \id_{A_p}$, and if $p \leq q \leq u$ then $f_{q , u} \circ f_{p , q} = f_{p , u}$. A \emph{direct limit} of a direct system $\{ A_p , f_{p , q} \}$ in $\mathcal{C}$ is an object $C$ of $\mathcal{C}$ together with a family of morphisms $\{g_p : A_p \to C \}_{p \in P}$ satisfying $g_q \circ f_{p , q} = g_p$, as well as the following universal property: if $D$ is an object of $\mathcal{C}$ and $\{h_p : A_p \to D \}_{p \in P}$ is a family of morphisms satisfying $h_q \circ f_{p , q} = h_p$, then there exists a unique morphism $d: C \to D$ such that $d \circ g_p = h_p$ for all $p \in P$. Note that (confusingly) a direct limit is an example of the category theoretic notion of a colimit (not a limit, as one might hope). In this paper, all directed sets considered will be $\nn$ with the usual order, and if $\{ A_n , f_{n , m} \}$ is a direct system indexed by $\nn$, we typically only list the maps of the form $f_n \coloneq f_{n , n + 1}$, as $f_{n , m}$ can be obtained by composing maps of this form. 

Let $X$ be a countable set, $R = k[X]$ a polynomial ring, and $X_1 \subseteq X_2 \subseteq \ldots \subseteq X$ an increasing sequence of finite subsets such that $\bigcup_{n \in \nn} X_n = X$. Denote by $R_n$ the polynomial ring $k[X_n]$, by $\eta_n : R_n \to R$ and $\iota_n : R_n \to R_{n + 1}$ the inclusion maps, and $\pi_n : R_{n + 1} \to R_n$ the projection map which sends $x \mapsto x$ if $x \in X_n$ and $x \mapsto 0$ otherwise. Let $I_n$ be an ideal in $R_n$. We call this \emph{situation} \hypertarget{situation: ast}{$(\ast)$}. 

\begin{lemma}\label{lem: ideal direct limit}
      In situation \hyperlink{situation: ast}{$(\ast)$}, if $\iota_n(I_n) \subseteq I_{n + 1}$ for all $n \in \nn$, then $R / I$ is a direct limit of the direct system $\{ R_n / I_n , \overline{\iota}_n \}_{n \in \nn}$, where $I = \bigcup_{n = 1}^\infty  \eta_n(I_n)R$.
\end{lemma}

\begin{proof}
First of all, notice that for any $n\in\nn$, the following diagram commutes. 

\begin{center}
\begin{tikzcd}
& R & 
\\ R_n \arrow[ru, "\eta_n"] \arrow[rr, "\iota_n" '] & & R_{n + 1} \arrow[lu, "\eta_{n + 1}" ']
\end{tikzcd}
\end{center}

This implies that $\eta_n(I_n) \subseteq \eta_{n + 1}(I_{n + 1})$, and so $I = \bigcup_{n = 1}^\infty  \eta_n(I_n)R$ is a nested union of ideals of $R$, and is therefore an ideal. Furthermore, since $\eta_n(I_n) \subseteq I$, and by hypothesis $\iota_n(I_n) \subseteq I_{n + 1}$, taking quotients yields another commutative diagram. 

\begin{center}
\begin{tikzcd}
& R / I & 
\\ R_n / I_n \arrow[ru, "\overline{\eta}_n"] \arrow[rr, "\overline{\iota}_n" '] & & R_{n + 1} / I_{n + 1} \arrow[lu, "\overline{\eta}_{n + 1}" ']
\end{tikzcd}
\end{center}

% et cetera

Now, suppose that there is another ring $S$ and maps $\overline{\mu}_n : R_n / I_n \to S$ such that $\overline{\mu}_{n + 1} \circ \overline{\iota}_n = \overline{\mu}_n$. Composing with the projection $p_n : R_n \to R_n / I_n$, we obtain maps $\mu_n \coloneq \overline{\mu}_n \circ p_n : R_n \to S$ which satisfy $\mu_{n + 1} \circ \iota_n = \overline{\mu}_{n + 1} \circ p_{n + 1} \circ \iota_n = \overline{\mu}_{n + 1} \circ \overline{\iota}_n \circ p_n = \overline{\mu}_n \circ p_n = \mu_n$. Note that $R$ is a direct limit of the direct system $\{ R_n , \iota_n \}_{n \in \nn}$, hence there exists a unique map $\varphi : R \to S$ such that $\varphi \circ \eta_n = \mu_n$ for all $n \in \nn$. We claim that $I \subseteq \ker \varphi$. Indeed, given any $a \in I$, by definition, there exists some $n \in \nn$ such that $a \in \eta_n(I_n)R$, and so $a = \sum_{i} r_i \eta_n(a_i) $ where $a_i \in I_n$. Applying $\varphi$, we obtain: 
\begin{equation*}
\varphi(a) = \varphi\!\left(\sum_{i} r_i \eta_n(a_i)\!\right) = \sum_{i} r_i \varphi(\eta_n(a_i)) \overset{(1)}{=} \sum_{i} r_i \mu_n(a_i) \overset{(2)}{=} \sum_{i} r_i \overline{\mu}_n(p_n(a_i)) \overset{(3)}{=} 0
\end{equation*}

(1) holds because $\varphi \circ \eta_n = \mu_n$. (2) holds by definition of $\mu_n$. (3) holds because $p_n(a_i) = 0$. Hence, we obtain a map $\overline{\varphi} : R / I \to S$ which satisfies $\overline{\varphi} \circ p = \varphi$, where $p : R \to R / I$ is the projection map. 

We now show that this map satisfies $\overline{\varphi} \circ \overline{\eta}_n = \overline{\mu}_n$. Indeed, we compute:
\begin{equation*}
\overline{\varphi} \circ \overline{\eta}_n \circ p_n = \overline{\varphi} \circ p \circ \eta_n = \varphi \circ \eta_n = \mu_n = \overline{\mu}_n \circ p_n.
\end{equation*}
Since $p_n$ is surjective, we may right-cancel it to obtain $\overline{\varphi} \circ \overline{\eta}_n = \overline{\mu}_n$. 

Now suppose that $\overline{\psi} : R / I \to S$ is another map which satisfies $\overline{\psi} \circ \overline{\eta}_n = \overline{\mu}_n$. Composing on the right with $p_n$ yields $\overline{\psi} \circ \overline{\eta}_n \circ p_n = \overline{\mu}_n \circ p_n$, which implies $\overline{\psi} \circ p \circ \eta_n = \mu_n$. However, $\varphi$ is the unique map which satisfies $\varphi \circ \eta_n = \mu_n$ for all $n \in \nn$. Thus, $\overline{\psi} \circ p = \varphi = \overline{\varphi} \circ p$, and because $p$ is surjective, it may be right canceled, yielding $\overline{\psi} = \overline{\varphi}$ as desired.
\end{proof}

%\begin{lemma}\label{lem: union of complexes and ideals}
%Let $\Delta$ be a simplicial complex so that $V(\Delta)$ is countable. Suppose there exists an increasing sequence $\Delta_1 \subseteq \Delta_2 \subseteq \ldots$ of finite full subcomplexes of $\Delta$. Then $\bigcup_{n \in \nn} \Delta_n = \Delta$ if and only if $\bigcup_{n \in \nn} V(\Delta_n) = V(\Delta)$ and $\bigcup_{n \in \nn} \eta_n(I_{\Delta_n , V(\Delta_n)}) K[\Delta] = I_{\Delta, k[\Delta]}$.
%\end{lemma}

\begin{corollary}\label{cor: stanley-resiner direct limit}
Let $\Delta$ be a simplicial complex so that $V(\Delta)$ is countable. Suppose there exists an increasing sequence $\Delta_1 \subseteq \Delta_2 \subseteq \ldots$ of finite full subcomplexes of $\Delta$ such that $\bigcup_{n \in \nn} \Delta_n = \Delta$. Then, the Stanley-Reisner ring $k[\Delta]$ is a direct limit of the Stanley-Reisner rings $k[\Delta_n]$. 
\end{corollary}

\begin{proof}
    Because $\bigcup_{n \in \nn} \Delta_n = \Delta$, it follows that $\bigcup_{n \in \nn} V(\Delta_n) = V(\Delta)$. Note that for any $n \in \nn$, since both $\Delta_n$ and $\Delta_{n + 1}$ are full subcomplexes of $\Delta$, and $\Delta_n$ is a subcomplex of $\Delta_{n + 1}$, it follows immediately that $\Delta_n$ is also a full subcomplex of $\Delta_{n + 1}$. Therefore, by Proposition \ref{prop: maps up and down full complexes}, we have that $\iota_n(I_{\Delta_n , V(\Delta_n)}) \subseteq I_{\Delta_{n + 1} , V(\Delta_{n + 1})}$ and $\pi_n ( I_{\Delta_{n + 1} , V(\Delta_{n + 1})} ) \subseteq I_{\Delta_n , V(\Delta_n)}$. By Lemma \ref{lem: ideal direct limit}, we have that $k[V(\Delta)] / I$ is the direct limit of the direct system $\{ k[\Delta_n] \coloneq k[V(\Delta_n)]/I_{\Delta_n , V(\Delta_n)} , \overline{\iota}_n \}$, where $I = \bigcup_{n \in \nn} \eta_n(I_{\Delta_n , V(\Delta_n)}) k[V(\Delta)]$. 
    
    We claim that $I = I_{\Delta , V(\Delta)}$, and showing this will conclude the proof. Indeed, note that each term $\eta_n(I_{\Delta_n , V(\Delta_n)}) k[V(\Delta)]$ of the union is generated by $\textbf{x}^A$ where $A \subseteq V(\Delta_n)$ is a finite set which is not a face of $\Delta_n$. Given such a set $A$, since $\Delta_n$ is a full subcomplex of $\Delta$, $A$ is also not a face of $\Delta$. Hence, $\textbf{x}^A \in I_{\Delta , V(\Delta)}$. Conversely, $I_{\Delta , V(\Delta)}$ is generated by $\textbf{x}^B$ where $B \subseteq V(\Delta)$ is a finite set which is not a face of $\Delta$. Given such a set $B$, since $\bigcup_{n \in \nn} V(\Delta_n) = V(\Delta)$, there must exist some $n \in \nn$ such that $B \subseteq V(\Delta_n)$. Because $\Delta_n$ is a full subcomplex of $\Delta$, it cannot be that $B$ is a face of $\Delta_n$. Therefore, $\textbf{x}^B \in I_{\Delta_n, V(\Delta_n)}$. 
\end{proof}

\begin{proof}[Proof of Theorem \ref{thm: main thm}]
By Proposition \ref{prop: maps up and down full complexes}, for each $n \in \nn$, there exist maps $\overline{\iota}_n: k[\Delta_n] \to k[\Delta_{n + 1}]$ and $\overline{\pi}_n: k[\Delta_{n + 1}] \to k[\Delta_{n}]$ such that $\overline{\pi}_n \circ \overline{\iota}_n = \id_{k[\Delta_n]}$. Since $k[\Delta_n]$ is Cohen-Macaulay, it is pure by Remark \ref{rmk: CM is pure}. So, by Corollary \ref{cor: countably many compatible linear sops}, for each $n \in \nn$, there exists a linear system of parameters $\theta_{n , 1}, \ldots, \theta_{n , d_n}$ such that $\overline{\pi}_{n + 1}(\theta_{n + 1 , j}) = \theta_{n , j}$ for $1 \leq j \leq d_n$. Furthermore, each $k[\Delta_n]$ is Noetherian Cohen-Macaulay by hypothesis, and so by Proposition \ref{prop: inclusion maps are flat}, the maps $\overline{\iota}_n$ are flat. Additionally, by Corollary \ref{cor: stanley-resiner direct limit}, $k[\Delta]$ is a direct limit of the direct system $\{ k[\Delta_n] , \overline{\iota}_n \}$.
\end{proof}

\begin{proof}[Proof of Theorem \ref{thm: main thm initial rings}]
Using the infinite Stanley-Reisner bijection (Proposition \ref{prop: infinite SR bijection}), it is clear that Theorems \ref{thm: main thm} and \ref{thm: main thm initial rings} are equivalent. 
\end{proof}

%%%%%%%%%%%%%%%%%%%%%%%%%%%%%%%%%%%%%%%%%%%%%%%%%%%%%%%%%%%%%%%%%%%%%%%%%%%

\section{Applications to Initial Ideals}\label{sec: Applications to Initial Ideals}

%{\color{orange}{\textbf{Notes:} Add a definition or two about initial ideals and grobner basis.}}

We now recall some background on Gr\"{o}bner bases (for a reference, see \cite[Chapter 15]{eisenbud2013commutative}). A \emph{monomial order} on a polynomial ring $k[x_1, \ldots, x_n]$ with finitely many variables is a total order $<$ of the monomials which satisfies (MO1) $1 \leq \textbf{x}^{(a_i)}$ for any monomial $\textbf{x}^{(a_i)}$, and (MO2) $\textbf{x}^{(a_i)} < \textbf{x}^{(b_i)}$ implies $\textbf{x}^{(a_i)} \textbf{x}^{(c_i)} < \textbf{x}^{(b_i)}  \textbf{x}^{(c_i)}$. Note that here $(a_i)$ denotes an element of $\zz^n$. Since $k[x_1, \ldots, x_n]$ is Noetherian, these conditions imply that $<$ is a well-order (i.e. every nonempty subset of monomials has a minimal element).

In \cite{iima2009grobner}, Iima and Yoshino give a definition of a monomial order and Gr\"{o}bner basis for an ideal in a polynomial ring with countably many variables $R\coloneq k[x_1, x_2, \ldots ]$ (\cite[Definition 1.1]{iima2009grobner}). In this circumstance, a monomial order is required to satisfy (MO1) and (MO2). However, since $R$ is not Noetherian, these do not necessarily imply that the order is well-ordered, and so it must be required separately that (MO3) $<$ is a well-order. For example, consider $k[x_1, x_2, \ldots]$ and order $x_1 < x_2 < \ldots$. Take the lexicographical order where $\textbf{x}^{(a_1, a_2, \ldots)} < \textbf{x}^{(b_1, b_2, \ldots)}$ if and only if $a_i < b_i$ for the maximal $i$ for which $a_i \neq b_i$. Note that such an $i$ must exist since $a_j = b_j = 0$ for all but finitely many $j$.

Let $f \in R$. The \emph{leading monomial} of $f$, denoted $\text{lm}_<(f)$, is the largest (with respect to $<$) monomial involved in $f$. The \emph{leading term} of $f$, denoted $\text{lt}_<(f)$, is the largest term of $f$, so that $\text{lt}_<(f) = c \text{lm}_<(f)$ for some $c \in k$. The \emph{initial ideal} with respect to $<$ of any subset $H \subseteq R$ is the monomial ideal $\text{in}_<(H)$ generated by the leading monomials (equivalently the leading terms) of all elements of $H$. If $G$ is a subset of some ideal $I \subseteq R$ and $f \in R$ is any polynomial, then the \emph{division algorithm} (\cite[Proposition 1.9]{iima2009grobner} implies that we can write 
\begin{equation*}
    f = \sum_i h_i g_i + r,
\end{equation*}
where $h_i \in R$, $g_i \in G$, none of the monomials involved in $r$ are in $\langle \text{lt}_<(g) \mid g \in G \rangle$, and $\text{lt}_<(h_i g_i) \leq \text{lt}_<(f)$. The element $r$ is called a \emph{remainder} of $f$ upon division by $G$. 

A \emph{Gr\"{o}bner basis} for an ideal $I \subseteq R$ is a subset $G \subseteq I$ such that the initial ideal $\text{in}_{<}(G)$ of $G$ is equal to the initial ideal $\text{in}_<(I)$ of $I$ (\cite[Definition 1.4]{iima2009grobner}). As in the case of finitely many variables, this implies that $G$ generates $I$. The \emph{$S$-polynomial} of $f , g \in R$ with respect to the term order $<$ is defined to be the following. 
\begin{equation*}
    S_{<}(f , g) = \frac{\text{lcm}( \text{lm}_{<}(f) , \text{lm}_{<}(g))}{\text{lt}_{<}(f)} f - \frac{\text{lcm}( \text{lm}_{<}(f) , \text{lm}_{<}(g))}{\text{lt}_{<}(g)} g
\end{equation*}

Note that $S_{<}(f , g)$ is always an element of $R$. The analog of Buchberger's Criterion for the infinite polynomial ring $R$ is the following result (\cite[Proposition 1.13]{iima2009grobner}).  

\begin{proposition}[Buchberger's Criterion]\label{Buchberger}
    A subset $G \subseteq R$ is a Gr\"{o}bner basis (for the ideal $\langle G \rangle$) if and only if for each $g , h \in G$, $0$ is a remainder of $S_{<}(g , h)$ upon division by $G$. 
\end{proposition}

We now apply Buchberger's Criterion to describe how Gr\"{o}bner bases behave with respect to inclusions of polynomial rings with countably many variables. Let $X \subseteq Y$ be finite or countable sets, $R \coloneq k[X]$ and $S \coloneq k[Y]$ polynomial rings, $\iota : R \to S$ the inclusion map which sends $x \mapsto x$, and $\pi : S \to R$ the projection map which sends $y \mapsto y$ if $y \in X$ and $y \mapsto 0$ otherwise. We say that a pair of term orders $<_R$ on $R$ and $<_S$ on $S$ is \emph{compatible} with $\iota$ (respectively, $\pi$) if for all $f \in R$, we have $\iota(\text{lt}_{<_R}(f)) = \text{lt}_{<_S}(\iota(f))$, (respectively, for all $g \in S$ with $\text{lt}_{<_S}(g) \in \im \iota$, we have that $\pi(\text{lt}_{<_S}(g)) = \text{lt}_{<_R}(\pi(g)$).

%Let $m \leq n$ be positive integers, $R \coloneq k[y_1, \ldots, y_m]$, $S \coloneq k[y_1, \ldots, y_n]$, $\iota : R \to S$ the inclusion map which sends $y_i \mapsto y_i$, and $\pi : S \to R$ the projection map which sends $y_i \mapsto y_i$ if $1 \leq i \leq m$ and $y_i \mapsto 0$ otherwise. We say that a pair of term orders $<_R$ on $R$ and $<_S$ on $S$ is \emph{compatible} with the maps $\iota$ and $\pi$, respectively, if for all $f \in R$ we have $\iota(\text{in}_{<_R}(f)) = \text{in}_{<_S}(\iota(f))$, and for all $g \in S$ with $\text{in}_{<_S}(g) \in \im \iota$ we have that $\pi(\text{in}_{<_S}(g)) = \text{in}_{<_R}(\pi(g))$.

\begin{lemma}\label{lem: comtainment of ideals implies containment of initials}
    Let $I \subseteq R$ and $J \subseteq S$ be ideals. 
    \begin{enumerate}
        \item If $\iota(I) \subseteq J$ and the pair of term orders $<_R$ and $<_S$ is compatible with $\iota$, then $\iota( \text{in}_{<_R}(I) ) \subseteq \text{in}_{<_{S}}(J)$.
        \item If $\pi(J) \subseteq I$ and the pair of term orders is compatible with $\pi$, then $\pi( \text{in}_{<_S}(J) ) \subseteq \text{in}_{<_R}(I)$.
    \end{enumerate}  
\end{lemma}

\begin{proof}

In case (1), given $f \in I$, $\iota(\text{lt}_{<_R}(f))=\text{lt}_{<_S}(\iota(f))$, since $<_R$ and $<_S$ are compatible with $\iota$. By hypothesis, we have that $\iota (I) \subseteq J$, and so $\iota(f) \in J$. Therefore, $\text{lt}_{<_S}(\iota(f)) \in \text{in}_{<_S}(J)$, and the claim follows. The proof of case (2) follows in a similar manner.
%let $g \in J$. Then, $\pi(\text{in}_{<_S}(g))=\text{in}_{<_R}(\pi(g))$, since $<_R$ and $<_S$ are compatible with $\pi$. By hypothesis, we have that, $\pi(J) \subseteq I$, and so $\pi(g)\in I$, implying that $\text{in}_{<_R}(\pi(g)) \in \text{in}_{<_R}(I)$. Once again, the claim easily follows.
\end{proof}

\begin{lemma}\label{lem: compatible orders and initial ideals}
    If $I \subseteq R$ is any ideal and $<_R$ and $<_S$ are compatible with $\iota$, then $\iota(\text{in}_{<_R}(I)) S = \text{in}_{<_S}(\iota(I)S)$. 
\end{lemma}

\begin{proof}
Let $G$ be a Gr\"{o}bner basis for $I$. We could, for example, take $G = I$. Then, by Proposition \ref{Buchberger}, for any $f, g \in G$, we can write $S_{<}(f , g) = \sum_i f_i g_i$ where $f_i \in R$ and $g_i \in G$. 
%{\color{blue} This should be in terms of $h$ throughout.} 
Since $<_R$ and $<_S$ are compatible with $\iota$, we have that $\text{lm}_{<R}(f) = \text{lm}_{<S}(f)$ and $\text{lt}_{<_R}(f) = \text{lt}_{<_S}(f)$ (and similarly for $g$). Thus, applying $\iota$, we obtain: 
%{\color{blue} Are we defining $S_{<_R}$ here?}
\begin{align*}
    \iota(S_{<R}(f , g)) &= \frac{\text{lcm}( \text{lm}_{<_R}(f) , \text{lm}_{<_R}(g))}{\text{lt}_{<_R}(f)} f - \frac{\text{lcm}( \text{lm}_{<_R}(f) , \text{lm}_{<_R}(g))}{\text{lt}_{<_R}(g)} g 
    \\&= \frac{\text{lcm}( \text{lm}_{<_S}(f) , \text{lm}_{<_S}(g))}{\text{lt}_{<_S}(f)} f - \frac{\text{lcm}( \text{lm}_{<_S}(f) , \text{lm}_{<_S}(g))}{\text{lt}_{<_S}(g)} g
    \\&= S_{<S}(f , g).
\end{align*}  

On the other hand, we have $\iota(S_{<_R}(f , g)) = \iota( \sum_i f_i g_i ) = \sum_i f_i g_i$ with $f_i \in S$ and $g_i \in G$. So, $0$ is a remainder of $\iota(S_{<_R}(f , g)) = S_{<S}(f , g)$ with respect to $G$ in $S$ as well. Hence, once again, by Proposition \ref{Buchberger}, $\iota(G)$ is a Gr\"{o}bner basis in $S$ for the ideal $\iota(G)S = \iota(I)S$. So, we can compute as follows: 
\begin{equation*}
    \iota(\text{in}_{<_R}(I))S \overset{(1)}{=} \iota(\text{in}_{<_R}(G))S \overset{(2)}{=} \iota(\{ \text{lm}_{<_R}(f) \mid f \in G \})S \overset{(3)}{=} \{ \text{lm}_{<_S}(\iota(f)) \mid f \in G \} S  \overset{(4)}{=} \text{in}_{<_S}(\iota(I)S). 
\end{equation*}

(1) follows because $G$ is a Gr\"{o}bner basis for $I$. (2) follows because for any ideal $J$ in any ring $T$ with generating set $H \subseteq J$, and for any ring homomorphism $\varphi : T \to T'$, we have that $\varphi(J) T' = \varphi(H) T'$. (3) follows because $<_R$ and $<_S$ are compatible with $\iota$. (4) follows because $\iota(G)$ is a Gr\"{o}bner basis for $\iota(I) S$, as we just proved. 
\end{proof}

One way to obtain compatible orders is to embed $R$ and $S$ into some larger polynomial ring with a monomial ordering, and then restrict this order. For the rest of this section, suppose we are in situation \hyperlink{situation: ast}{$(\ast)$}. Let $<$ be a monomial order on $R$ and $<_n$ be the restriction of $<$ to the set of monomials contained in $R_n$.

\begin{lemma}\label{lem: restricted orders are compatible}
For all $n \in \nn$, the pair of monomial orders $<_n$ and $<$ is compatible with the inclusion map $\eta_n$. The monomial orders $<_n$ and $<_{n + 1}$ are compatible with the inclusion map $\iota_n$ and the projection map $\pi_n$. 
\end{lemma}

\begin{proof}
Let $f \in I_n$ for some arbitrary $n \in \nn$. Because the monomial orders $<_n$ and $<_{n+1}$ are both restrictions of $<$, we have that $\text{lm}_<(\eta_n(f)) = \eta_n(\text{lm}_{<_n}(f))$ and $\text{lm}_{<_{n+1}}(\iota_n(f)) = \iota_n(\text{lm}_{<_n}(f))$. If $g \in I_{n + 1}$ and $\text{lm}_{<_{n + 1}}(g)$ involves only variables from $X_n$, then $\text{lm}_{<_{n + 1}}(g)$ is in particular larger than any other monomial of $g$ which only involves variables from $X_n$. Hence, $\pi_n(\text{lm}_{<_{n + 1}}(g)) = \text{lm}_{<_{n + 1}}(g) = \text{lm}_n(\pi_n(g))$ as desired.
\end{proof}

\begin{lemma}\label{lem: union of initials is initial}
    We have that $\text{in}_< ( \bigcup_{n \in \nn} \eta_n(I_n) R ) = \bigcup_{n \in \nn} \eta_n( \text{in}_{<_n} (I_n))R$. 
\end{lemma}

\begin{proof}
By Lemma \ref{lem: restricted orders are compatible}, the pair of orders $<_n$ and $<$ is compatible with the inclusion map $\eta_n$. Hence, by Lemma \ref{lem: compatible orders and initial ideals}, we have that $\eta_n( \text{in}_{<_n} (I_n))R = \text{in}_{<}(\eta_n(I_n)R)$. Thus, we have that $\bigcup_{n \in \nn} \eta_n( \text{in}_{<_n} (I_n))R = \bigcup_{n \in \nn} \text{in}_{<} (\eta_n(I_n)R)$, where the latter is clearly contained inside $\text{in}_< ( \bigcup_{n \in \nn} \eta_n(I_n) R )$. For the reverse containment, if $f \in \bigcup_{n \in \nn} \eta_n(I_n) R$, then there exists some $n \in \nn$ such that $f \in \eta_n(I_n) R$. Therefore, $\text{lm}_<(f) \in \text{in}_<( \eta_n(I_n) R )$. But, by definition, $\text{in}_< ( \bigcup_{n \in \nn} \eta_n(I_n) R )$ is generated by leading monomials of this form. So, $\text{in}_< ( \bigcup_{n \in \nn} \eta_n(I_n) R ) \subseteq \bigcup_{n \in \nn} \text{in}_{<} (\eta_n(I_n)R) = \bigcup_{n \in \nn} \eta_n( \text{in}_{<_n} (I_n))R$, as desired. 
\end{proof}

%\begin{proposition}\label{prop: ideal direct limit}
%     In situation $(\ast \ast \ast)$, $R / I$ is a direct limit of the direct system $\{ R_n / I_n , \overline{\iota}_n \}_{n \in \nn}$, where $I = \bigcup_{n = 1}^\infty  \eta_n(I_n)R$. 
%\end{proposition}

%\begin{proof}
%\hilight{finish!}
%\end{proof}

%\begin{corollary}\label{cor: initial ideal direct limit}
%    In situation $(\ast \ast \ast)$, $R / \text{in}_{<}(I)$ is a direct limit of the direct system $\{ R_n / \text{in}_{<_n}(I_n) , \overline{\iota}_n \}_{n \in \nn}$
%\end{corollary}

\begin{proof}[Proof of Theorem \ref{thm: main thm rings}]
By Lemma \ref{lem: restricted orders are compatible}, we have that $<_n$ and $<_{n + 1}$ are compatible with $\iota_n$ and $\pi_n$. Therefore, by Lemma \ref{lem: comtainment of ideals implies containment of initials}, we have that $\iota_n(\text{in}_{<_n}(I_n)) \subseteq \text{in}_{<_{n + 1}}(I_{n + 1})$ and $\pi_n(\text{in}_{<_{n + 1}}(I_{n + 1})) \subseteq \text{in}_{<_{n}}(I_{n})$. From Lemma \ref{lem: union of initials is initial}, it follows that $\bigcup_{n \in \nn} \eta_n( \text{in}_{<_n} (I_n))R=\operatorname{in}_{<}(I)$. The result follows from Theorem \ref{thm: main thm initial rings}.
\end{proof}

%%%%%%%%%%%%%%%%%%%%%%%%%%%%%%%%%%%%%%%%%%%%%%%%%%%%%%%%%%%%%%%%%%%%%%%%%%%

\section{Infinite Matrix Schubert Varieties} \label{sec: Infinite Matrix Schubert Varieties}

For $m \leq n$, let $S_{m , n}$ denote the set of injections $\{ 1, \ldots, m \} \to \{ 1, \ldots, n \}$. Elements of $S_{m , n}$ are called \emph{partial permutations}. Given any $\sigma \in S_{m , n}$, denote by $[\sigma]$ the $m \times n$ \textit{partial permutation matrix of $\sigma$}, which, by definition, has a $1$ in entry $(j , \sigma(j))$ and zeros elsewhere. Define the \textit{rank matrix of $\sigma$}, namely $r(\sigma)$, to be the $m \times n$ matrix whose $(i , j)$ entry is the number of $1$'s both weakly to the left and above the $(i , j)$ entry in $[\sigma]$. The \textit{Schubert determinantal ideal} $I_{\sigma}$ is defined to be the ideal of the polynomial ring $k[x_{i , j} \mid 1 \leq i \leq m, \, 1 \leq j \leq n]$ generated by the minors of size $r(\sigma)_{k , l} + 1$ (for all choices of $k\in[m]$ and $l\in[n]$) of the matrix of variables $[x_{i , j}]_{1 \leq i \leq k , \, 1 \leq j \leq l}$. The vanishing set of $I_{\sigma}$, denoted by $X_{\sigma}$, is called the \emph{matrix Schubert variety} corresponding to the partial permutation $\sigma$. By results of Fulton \cite[Proposition 3.3]{fulton_flags_1992}, $I_{\sigma}$ is a prime ideal, and thus $k[x_{i , j}] / I_{\sigma}$ is the coordinate ring of this variety. 

\begin{example}
    Let $\sigma=2531\in S_{4,5}$. Notice that 
    \begin{align*}
    [\sigma]=\begin{pmatrix}
    0 & 1 & 0 & 0 & 0 \\
    0 & 0 & 0 & 0 & 1\\
    0 & 0 & 1 & 0 & 0\\
    1 & 0 & 0 & 0 & 0
    \end{pmatrix},
    \hspace{10ex}
    r(\sigma)=\begin{pmatrix}
    0 & 1 & 1 & 1 & 1 \\
    0 & 1 & 1 & 1 & 2\\
    0 & 1 & 2 & 2 & 3\\
    1 & 2 & 3 & 3 & 4
    \end{pmatrix},
    \end{align*}
    and \begin{align*}I_{\sigma}=\langle x_{1,1},x_{2,1},x_{3,1},
    x_{1,3}x_{2,2}-x_{1,2}x_{2,3},
    x_{1,4}x_{2,2}-x_{1,2}x_{2,4},
    x_{1,4}x_{2,3}-x_{1,3}x_{2,4}\rangle.
    \end{align*}
\end{example}

%{\color{orange}{Should we put the definitions in the following paragraph below in definition environments?}}

Denote by $S_\infty$ the set of bijections from the natural numbers to the natural numbers. Suppose $\sigma \in S_\infty$. Then, for each $m \in \nn$, the restriction $\sigma|_{\{ 1 , \ldots , m \}}$ is a partial permutation in 
%{\color{orange}{$S_{\max \sigma([m]), m}$}}, 
$S_{m, \max \sigma([m])}$, 
namely the \emph{$m$th partial permutation of $\sigma$}, denoted by $\sigma_m$. The \emph{permutation matrix} of $\sigma_m$ is a %{\color{orange}{$\max \sigma([m]) \times m$}} 
$m \times \max \sigma([m])$
matrix with a $1$ entry in $(j, \sigma(j))$ and $0$'s elsewhere. Denote by $T_m$ the polynomial ring $k[x_{i , j} \mid 1 \leq i \leq m, 1 \leq j \leq \max \sigma([m])]$ with variables indexed by entries in a $m\times \max \sigma([m])$ matrix. Let $\iota_m : T_m \to T_{m + 1}$ and $\pi_m: T_{m + 1} \to T_m$ be the inclusion and projection maps. Define $T \coloneq k[x_{i , j} \mid i , j \in \nn]$.
%{\color{blue} Do we need to indicate (through some notation) that $R_m$ depends on $\sigma$?} 
%{\color{cyan} Sounds good! I think we should also change to $T$ instead of $R$ in this section since we used $R$ in the previous section to mean something else.}

\begin{example}\label{example:infiniteSchubIdeal}
Define $\sigma:\nn\to\nn$ such that
\begin{align*}
\sigma(m)=
\begin{cases}
1 &\text{if $m=1$},\\
m+1 &\text{if $m\in 2\nn$},\\
m-1 &\text{if $m\in 2\nn+1$}.
\end{cases}
\end{align*}
For each $\sigma_m$ where $m\in\nn$, we obtain the Schubert determinantal ideal
\begin{align*}
I_{\sigma_m}=\left\langle \det\left([x_{i , j}]_{1 \leq i,j \leq n}\right)\mid n\leq m\text{ and } n\in2\nn\right\rangle 
\end{align*}
in its respective ring $T_m$. Notice that $I_{\sigma_m}$ and $I_{\sigma_{m+1}}$ have the same generating set for $m\in 2\nn$.
\end{example}

A term order on $k[x_{i , j} \mid 1 \leq i \leq m, 1 \leq j \leq n]$ is called \emph{antidiagonal} if the initial term of any minor of the variable matrix $[x_{i , j}]_{1 \leq i \leq m , 1 \leq j \leq n}$ is the antidiagonal term. 
%{\color{blue} Must still update this next example.} 
%For example, the lexicographical order corresponds to the variable order $x_{1 , 1} > x_{2 , 1} > \ldots > x_{n , 1} > x_{1 , 2} > x_{2 , 2} > \ldots $ where the largest entry is in the top left, each column is decreasing downward, and all entries in a given column are larger than all entries in a subsequent column.
%For example, the variable order $x_{n , 1} > x_{n , 2} > \ldots > x_{n , n} > x_{n-1 , 1} > x_{n-2 , 2} > \ldots $ where the largest entry is in the bottom left, each row is decreasing leftward, and all entries in a given row are larger than all entries in an earlier row. {\color{blue} Paraphrasing MS pg323. This is okay for any partial permutation but not for the full permutation?}
Work of Knutson, Miller, and Sturmfels shows that the minors defining $I_{\sigma}$ form a Gr\"{o}bner basis with respect to any antidiagonal term order. 

For the remainder of this paper, we fix an antidiagonal term order $<$. Denote by $\text{in}_<(I_\sigma)$ the initial ideal of $I_{\sigma}$ with respect to this order. Note that this is a square-free mononial ideal. Let $\Delta_\sigma$ denote the simplicial complex associated to $\text{in}_<(I_\sigma)$.

\begin{lemma}\label{lem: iota and pi preserve schubert ideals}
    $\iota_m(I_{\sigma_m}) \subseteq I_{\sigma_{m + 1}}$ and $\pi_m(I_{\sigma_{m + 1}}) \subseteq I_{\sigma_{m}}$.
\end{lemma}

\begin{proof}
Recall that $I_{\sigma_m}$ is generated by minors of size determined by the rank matrix $r(\sigma_m)$. Suppose $f$ is one of these generators. Then, there exists some $1 \leq i \leq m$ and $1 \leq j \leq \max \sigma([m])$ such that $f$ is a minor of size $r + 1$ of the submatrix of variables $[x_{\alpha , \beta}]_{\, 1 \leq \alpha \leq i ,\, 1 \leq \beta \leq j}$, where $r = r(\sigma_m)_{i , j}$. Since $[\sigma_m]$ is an upper left submatrix of 
%$\sigma_{m + 1}$ 
%{\color{blue}
$[\sigma_{m+1}]$, the $(i , j)$ entry of the rank matrix $r(\sigma_{m + 1})$ is also $r$. So, $f$ is also a generator of $I_{\sigma_{m + 1}}$. Therefore, $\iota_m(f) = f \in I_{\sigma_{m + 1}}$, and we obtain the first desired inclusion. 

Conversely, let $f$ be a minor generator of $I_{\sigma_{m + 1}}$. Then, as above, there exists some $1 \leq i \leq m+1$ 
%{\color{blue}Should be $m+1$?} 
and $1 \leq j \leq \max \sigma([m + 1])$ such that $f$ is a minor of a submatrix $U$ of size $r + 1$ in the variable matrix $X_{m + 1} \coloneq [x_{\alpha , \beta} \mid 1 \leq \alpha \leq m + 1,\, 1 \leq \beta \leq 
%{\color{blue}j \max \sigma([m + 1])}]$, 
\max \sigma([m + 1])]$, 
where $r$ is the $(i , j)$ entry of the rank matrix $r(\sigma_{m + 1})$. Note that, if the bottom row of $U$ has index $m+1$ or if the last column of $U$ has index strictly greater than $\max \sigma([m])$, then expanding $f$ along the bottom row or down the last column, respectively, shows that $\pi_m(f) = 0 \in I_{\sigma_m}$. Let us call this situation \hypertarget{situation: dagger}{($\dagger$)}. We now consider several cases. 

Either $\sigma(m+1) < \max \sigma([m])$ or $\sigma([m]) < \sigma(m+1)$. In the first situation, there are five cases. In the second situation, there are two more cases to consider.

\NiceMatrixOptions
{
    custom-line ={command= H, tikz= dashed, width= 1mm}, % horizontal, dashed
    custom-line = {letter= I, tikz= dashed, width= 1mm}, % vertical, dashed
}
\begin{figure}[htp]
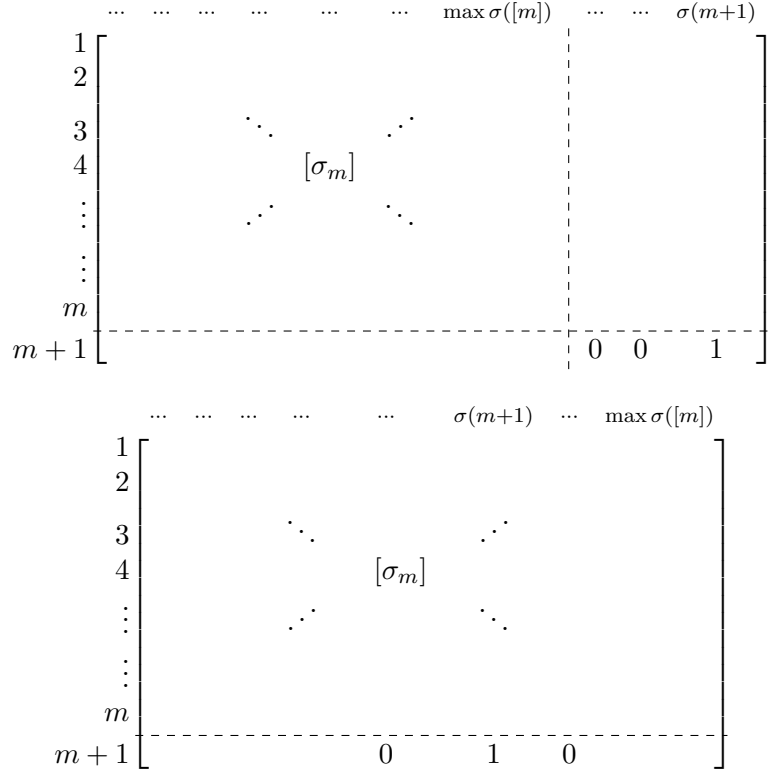

\begin{equation*}
\begin{bNiceArray}{cccccccIccc}[
  first-row,code-for-first-row=\scriptstyle,
  first-col,code-for-last-col=\scriptstyle,
]
        & \cdots & \cdots & \cdots & \cdots & \cdots & \cdots & \max \sigma([m]) & \cdots & \cdots & \sigma(m+1)\\
1       &   &   &           &               &   &   &   &   &   &   \\
2       &   &   &           &               &   &   &   &   &   &   \\
3       &   &   &     &      \ddots         &   &  \iddots &   &   &   &   \\
4       &   &   &           &     & [\sigma_m] &   &   &   &   &   \\
\vdots  &   &   &     &    \iddots           &   & \ddots &   &   &   &   \\
\vdots  &   &   &           &               &   &   &  &   &   &   \\
m       &   &   &           &               &   &   &   &  &   &  \\ \H
m+1     &   &   &           &               &   &   &   & 0  & 0 & 1 \\
\end{bNiceArray}
\end{equation*}

\begin{equation*}
\begin{bNiceArray}{cccccccc}[
  first-row,code-for-first-row=\scriptstyle,
  first-col,code-for-last-col=\scriptstyle,
]
  & \cdots  & \cdots & \cdots & \cdots & \cdots & \sigma(m+1) & \cdots & \max \sigma([m])\\
1 &   &   &   &   &   &   &  &  \\
2 &   &   &   &   &   &   &  & \\
3 &   &   &   & \ddots  &   &  \iddots &   &  \\
4 &   &   &   &   & \quad[\sigma_m] &   &  &  \\
\vdots &   &   &   &  \iddots &  & \ddots &  &  \\
\vdots &   &   &   &   &   &   &  & \\
m &   &   &   &   &   &   &  &  \\ \H
m+1 &   &   &   &   &  0 &  1 & 0 & \\
\end{bNiceArray}
\end{equation*}
\caption{Pictoral description of casework.}
\end{figure}

%{\color{blue} Overkill to replace the situation labels with clickable links? I've found myself scrolling to find the situation a few times. For example, we call this situation \hypertarget{situation: dagger}{($\dagger$)}}

%{\color{blue} We would then place \hyperlink{situation: dagger}{($\dagger$)} where referencing.}

\begin{enumerate}[label={(\emph{Case \arabic*})}]
\item Suppose $1 \leq i \leq m$ and $1 \leq j \leq \max \sigma([m])$. Then, $U$ is contained in $X_m$, and so there is an $r$ in the $(i , j)$ entry of the rank matrix $r(\sigma_m)$ as well. Therefore, $f$ is also a generator of $I_{\sigma_m}$. It must be that $\pi_m(f)=f \in I_{\sigma_m}$.

\item Suppose $1 \leq i \leq m$ and $j > \max \sigma([m])$ (in particular, $\sigma(m+1)>\max \sigma([m])$). Then, there is an $r$ in the $(i , \max \sigma([m]))$ entry of the rank matrix $r(\sigma_m)$ too. So, either we are in situation \hyperlink{situation: dagger}{($\dagger$)}, or $U$ is contained in $X_m$, and so $f$ is also a generator of $I_{\sigma_m}$, implying that $\pi_m(f)=f \in I_{\sigma_m}$.

\item\label{case 3} Suppose that $i = m + 1$, $\sigma(m + 1) > \max \sigma([m])$, and $1 \leq j \leq \max \sigma([m])$. Then, there is an $r$ in the $(i - 1 , j)$ entry of the rank matrix $r(\sigma_m)$. We must be in situation \hyperlink{situation: dagger}{($\dagger$)}, or $U$ is contained in $X_m$, and so $f$ is also a generator of $I_{\sigma_m}$, implying $\pi_m(f) =f\in I_{\sigma_m}$.

%(\emph{Case 6}). Suppose that $i = m + 1$, that $\sigma(m + 1) < \max \sigma([m])$, and that $\max \sigma([m]) < j$. Then there is a $r - 1$ in the $(i - 1 , \max \sigma([m]))$ entry of the rank matrix of $\sigma_m$ and the desired result follows from a similar argument to Case 5. {\color{blue}This case is not possible, the maximum column index is $\max \sigma([m])$.} {\color{orange}{Agreed!}}

\item Suppose that $i = m + 1$, $\sigma(m + 1) > \max \sigma([m])$, and $\max \sigma([m]) < j <\sigma(m+1)$. Then, there is a $r$ in the $(i - 1 , \max \sigma([m]))$ entry of the rank matrix $r(\sigma_m)$, and the desired result follows from a similar argument to \ref{case 3}.

\item\label{case 5} Suppose that $i = m + 1$, $\sigma(m + 1) > \max \sigma([m])$, and $j = \sigma(m + 1)$. Then, there is a $r - 1$ in the $(i - 1 , \max \sigma([m]))$ entry of the rank matrix of $\sigma_m$. So, either we are in situation \hyperlink{situation: dagger}{($\dagger$)}, or $U$ is contained in $X_m$. If the latter is true, expanding along any row or column shows that $f$ is a $T_m$-linear combination of minors of size $r$, which are in the upper left submatrix of $X_m$ with lower right corner $(i-1,j)$ (and are therefore generators of $I_{\sigma_m}$). This implies that $\pi_m(f) =f\in I_{\sigma_m}$ as well.

\item Suppose that $i = m + 1$, $\sigma(m + 1) < \max \sigma([m])$, and $1 \leq j < \sigma(m + 1)$. The result then follows exactly as in \ref{case 3}. 

\item Suppose $i = m + 1$, $\sigma(m + 1) < \max \sigma([m])$, and $\sigma(m + 1) \leq j \leq \max \sigma([m])$. Then, there is a $r - 1$ in the $(i - 1 , j)$ entry of the rank matrix $r(\sigma_m)$. So, either we are in situation \hyperlink{situation: dagger}{($\dagger$)}, or $U$ is contained in $X_m$, and the desired result follows as in \ref{case 5}. \qedhere
\end{enumerate}
\end{proof}

%{\color{orange}{What about the case when $i = m + 1$, $\sigma(m + 1) > \max \sigma([m])$, and $j > \sigma(m + 1)$? - JK. This case is not possible too?}}

\begin{lemma}\label{lem: partial det ideals approximate big det ideal}
    We have that $\bigcup_{n \in \nn} \eta_n(I_{\sigma_n}) T = I_\sigma$. 
\end{lemma}

\begin{proof}
    Recall from the proof of Lemma \ref{lem: iota and pi preserve schubert ideals} that $X_{m} = [x_{\alpha , \beta} \mid 1 \leq \alpha \leq m,\, 1 \leq \beta \leq 
    %{\color{blue}j \max \sigma([m + 1])}]$, 
    \max \sigma([m])]$. Suppose that $f \in \bigcup_{n \in \nn} \eta_n(I_{\sigma_n}) T$. Then, $f$ is a finite sum of terms of the form 
    $$f = \sum_{i=1}^M t_if_i$$ 
    for finite $M$, where $t_i \in T$ and $f_i \in \eta_{m_i}(I_{\sigma_{m_i}})T$. 
    Now, each $f_i$ is in the image of $\eta_{m_i}(I_{\sigma_{m_i}})$ for some $m_i$, so must be itself of the form 
    $$f_i = \sum_{j=1}^{N_i} r_jg_j$$
    for finite $N_i$, where $r_j \in T_{m_i}$ and $g_j \in I_{\sigma_{m_i}}.$
    %Without loss of generality, we may take each of the $f_i$ to be the image under $\eta_{m_i}$ of a generator of their respective $I_{\sigma_{m_i}}$ {\color{orange}{Why?}}. 
    That is, each $g_j$ is a minor of the variable matrix $X_{m_i}$. But then, $\eta_{m_i}(g_j) = g_j \in T$, too. In particular, each $g_j$ must be a minor generator of $I_\sigma$, the ideal of all such minors, hence $f \in I_\sigma$.

    Now, suppose that $h \in I_{\sigma}$. Again, $h$ must be a finite sum of terms $t_ih_i$ for $t_i \in T$ and minor generators $h_i \in I_\sigma$. Since each $h_i$ is a minor of some (finite) matrix of variables, say $X_{m'_i}$, we then have that $h_i = \eta_{m'_i}(h_i) \in \eta_{m'_i}(I_{m'_i})T$ for some $m'_i \in \nn$. Hence, $h \in \bigcup_{n \in \nn} \eta_n(I_{\sigma_n}) T.$
\end{proof}

\begin{definition}
    A term order on $T = k[x_{i , j} \mid i , j \in \nn]$ is called \emph{antidiagonal} if the restriction $<_m$ of $<$ to $T_m$ for each $m \in \nn$ is antidiagonal. 
\end{definition}

\begin{example}\label{example:infiniteSchubIdeal2}
Continuing with Example \ref{example:infiniteSchubIdeal} and using an antidiagonal term order, we have that
\begin{align*}
\text{in}_{<_m}(I_{\sigma_m})=\left\langle x_{1,n}x_{2,n-1}\ldots x_{n,1}\mid n\leq m\text{ and } n\in2\nn\right\rangle. 
\end{align*}
We see that $\Delta_{\sigma}$ is an infinite simplicial complex with no facets.
\end{example}

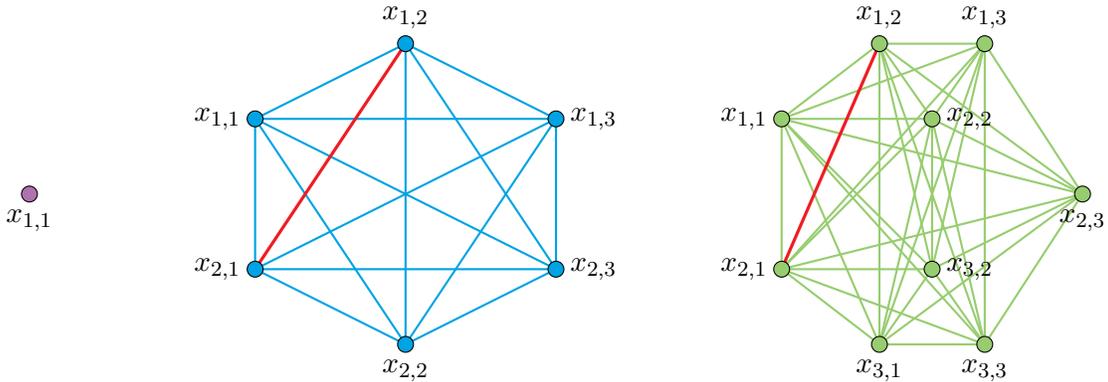
\begin{figure}[H]
\begin{center}
\begin{tikzpicture}
%C1 edges
%C1 vertices
\draw[fill=Orchid] (-5,0) circle (3pt);
%C1 labels for vertices
\node at (-5,-.35) {$x_{1,1}$};
%C2 edges
\draw[-, thick, Cerulean] (2,1) -- (0,2);
\draw[-, thick, Cerulean] (2,1) -- (-2,1);
\draw[-, thick, Cerulean] (2,1) -- (-2,-1);
\draw[-, thick, Cerulean] (2,1) -- (0,-2);
\draw[-, thick, Cerulean] (2,1) -- (2,-1);
\draw[-, thick, Cerulean] (0,2) -- (-2,1);
\draw[-, thick, Cerulean] (0,2) -- (0,-2);
\draw[-, thick, Cerulean] (0,2) -- (2,-1);
\draw[-, thick, Cerulean] (-2,1) -- (-2,-1);
\draw[-, thick, Cerulean] (-2,1) -- (0,-2);
\draw[-, thick, Cerulean] (-2,1) -- (2,-1);
\draw[-, thick, Cerulean] (-2,-1) -- (0,-2);
\draw[-, thick, Cerulean] (-2,-1) -- (2,-1);
\draw[-, thick, Cerulean] (0,-2) -- (2,-1);
\draw[-, very thick, Red] (-2,-1) -- (0,2);
%C2 vertices
\draw[fill=Cerulean] (2,1) circle (3pt);
\draw[fill=Cerulean] (0,2) circle (3pt);
\draw[fill=Cerulean] (-2,1) circle (3pt);
\draw[fill=Cerulean] (-2,-1) circle (3pt);
\draw[fill=Cerulean] (0,-2) circle (3pt);
\draw[fill=Cerulean] (2,-1) circle (3pt);
%C2 labels for vertices
\node at (-2.5,1) {$x_{1,1}$};
\node at (0,2.35) {$x_{1,2}$};
\node at (2.5,1) {$x_{1,3}$};
\node at (-2.5,-1) {$x_{2,1}$};
\node at (0,-2.35) {$x_{2,2}$};
\node at (2.5,-1) {$x_{2,3}$};
%C3 edges
\draw[-, thick, YellowGreen] (5,1) -- (6.3,2);
\draw[-, thick, YellowGreen] (5,1) -- (7.7,2);
\draw[-, thick, YellowGreen] (5,1) -- (5,-1);
\draw[-, thick, YellowGreen] (5,1) -- (7,1);
\draw[-, thick, YellowGreen] (5,1) -- (9,0);
\draw[-, thick, YellowGreen] (5,1) -- (6.3,-2);
\draw[-, thick, YellowGreen] (5,1) -- (7,-1);
\draw[-, thick, YellowGreen] (5,1) -- (7.7,-2);

\draw[-, thick, YellowGreen] (6.3,2) -- (7.7,2);
\draw[-, thick, YellowGreen] (6.3,2) -- (7,1);
\draw[-, thick, YellowGreen] (6.3,2) -- (9,0);
\draw[-, thick, YellowGreen] (6.3,2) -- (6.3,-2);
\draw[-, thick, YellowGreen] (6.3,2) -- (7,-1);
\draw[-, thick, YellowGreen] (6.3,2) -- (7.7,-2);

\draw[-, thick, YellowGreen] (7.7,2) -- (5,-1);
\draw[-, thick, YellowGreen] (7.7,2) -- (7,1);
\draw[-, thick, YellowGreen] (7.7,2) -- (9,0);
\draw[-, thick, YellowGreen] (7.7,2) -- (6.3,-2);
\draw[-, thick, YellowGreen] (7.7,2) -- (7,-1);
\draw[-, thick, YellowGreen] (7.7,2) -- (7.7,-2);

\draw[-, thick, YellowGreen] (5,-1) -- (7,1);
\draw[-, thick, YellowGreen] (5,-1) -- (9,0);
\draw[-, thick, YellowGreen] (5,-1) -- (6.3,-2);
\draw[-, thick, YellowGreen] (5,-1) -- (7,-1);
\draw[-, thick, YellowGreen] (5,-1) -- (7.7,-2);

\draw[-, thick, YellowGreen] (7,1) -- (9,0);
\draw[-, thick, YellowGreen] (7,1) -- (6.3,-2);
\draw[-, thick, YellowGreen] (7,1) -- (7,-1);
\draw[-, thick, YellowGreen] (7,1) -- (7.7,-2);

\draw[-, thick, YellowGreen] (9,0) -- (6.3,-2);
\draw[-, thick, YellowGreen] (9,0) -- (7,-1);
\draw[-, thick, YellowGreen] (9,0) -- (7.7,-2);

\draw[-, thick, YellowGreen] (6.3,-2) -- (7,-1);
\draw[-, thick, YellowGreen] (6.3,-2) -- (7.7,-2);

\draw[-, thick, YellowGreen] (7,-1) -- (7.7,-2);

\draw[-, very thick, Red] (5,-1) -- (6.3,2);
%C3 vertices
%\draw[fill=Cerulean] (9,1) circle (3pt);
\draw[fill=YellowGreen] (7,1) circle (3pt); %22
\draw[fill=YellowGreen] (7.7,2) circle (3pt); %13
\draw[fill=YellowGreen] (6.3,2) circle (3pt); %12
\draw[fill=YellowGreen] (5,1) circle (3pt); %11
%\draw[fill=Cerulean] (9,-1) circle (3pt);
\draw[fill=YellowGreen] (7,-1) circle (3pt); %32
\draw[fill=YellowGreen] (7.7,-2) circle (3pt); %33
\draw[fill=YellowGreen] (6.3,-2) circle (3pt); %31
\draw[fill=YellowGreen] (5,-1) circle (3pt); %21
%\draw[fill=Cerulean] (7,.5) circle (3pt);
\draw[fill=YellowGreen] (9,0) circle (3pt); %23
%C3 labels for vertices
\node at (4.5,1) {$x_{1,1}$};
\node at (6.3,2.35) {$x_{1,2}$};
\node at (7.7,2.35) {$x_{1,3}$};
\node at (4.5,-1) {$x_{2,1}$};
%\node at (6.5,.5) {$x_{2,2}$};
\node at (9,-.35) {$x_{2,3}$};
%\node at (9.5,.8) {$x_{2,3}$};
\node at (7.5,1) {$x_{2,2}$};
\node at (6.3,-2.35) {$x_{3,1}$};
\node at (7.7,-2.35) {$x_{3,3}$};
%\node at (9.5,-.8) {$x_{3,3}$};
\node at (7.5,-1) {$x_{3,2}$};
\end{tikzpicture}
\caption{The $1$-skeletons for $\Delta_{\sigma_m}$ for $m\in[3]$ from Example \ref{example:infiniteSchubIdeal2}, respectively. A red edge in the complex denotes a missing edge (counterintuitively).}
\end{center}
\end{figure}

\begin{proposition}\label{prop: infinite antidiagonal term order}
    There exists an antidiagonal term order on $T = k[x_{i , j} \mid i , j \in \nn]$. 
\end{proposition}

\begin{proof}
We define an order on the variables in the matrix $[ x_{i , j} \mid i , j \in \nn]$ as follows. For $m \in \zz$, denote by $Q_m$ the $m$th diagonal $\{ x_{i , j} \mid j - i = m \}$, so that e.g. $Q_0$ is the diagonal and $Q_1$ is the super diagonal of the variable matrix. For a given $Q_m$, we order $x_{i , j} < x_{i + 1 , j + 1}$. Additionally, we order $Q_0 < Q_1 < Q_{-1} < Q_2 < Q_{-2} < \ldots$. Note that this is a well-order which has order-type the ordinal $\omega^2$. We now order the monomials of $k[x_{i , j} \mid i , j \in \nn]$ using the lexicographical order with respect to this variable order. 

To show that this is an antidiagonal term order, it suffices to show that for any minor in the variable matrix $[x_{i, j} \mid i , j \in \nn]$, the leading term is the antidiagonal term. We prove this by induction on minor size. The result is trivial for minors of size $1$. So, suppose the result holds for minors of size $n$. Let $i_1 < \ldots < i_{n + 1}$ and $j_1 < \ldots < j_{n + 1}$ be the row and column indices of a minor $M$ of size $n + 1$. First, suppose that $| j_1 - i_{n + 1} | \geq | j_{n + 1} - i_1 |$. Then, $x_{i_{n + 1} , j_1} > x_{i , j}$ for any $x_{i , j}$ involved in $M$. We expand the minor along the first column, so that it is equal to $M = x_{i_1 , j_1} M_1 \pm \ldots \pm x_{i_{n + 1} , j_1} M_{n + 1}$, where each $M_i$ is a minor of size $n$. Note that the only terms of $M$ which involve the variable $x_{i_{n + 1} , j_1}$ appear in the terms of $x_{i_{n + 1} , j_1} M_{n + 1}$. Since the term order is lexicographic, this implies that the leading term of $M$ must be a term of $x_{i_{n + 1} , j_1} M_{n + 1}$. By induction, the leading term of $M_{n + 1}$ is the antidiagonal term. Hence, the leading term of $M$ is also the antidiagonal. On the other hand, if $| j_1 - i_{n + 1} | < | j_{n + 1} - i_1 |$, then expanding along the top row and applying the same logic, the desired result again follows. 
%For minors of size $2$, suppose that $i_1 < i_2$ and $j_1 < j_2$ are the row and column indices of the minor respectively. Define $n_1 \coloneq j_1 - i_2$ and $n_2 \coloneq j_2 - i_1$ and suppose that $| n_1 | \geq | n_2 |$. Then we have that \hilight{explain} $x_{i_2 , j_1} >  x_{i_1 , j_2}$ and also that $x_{i_2 , j_1} > x_{i_2 , j_2}$ and $x_{i_1 , j_2} > x_{i_1 , j_1}$. Hence it follows that $x_{i_2 , j_1} x_{i_1 , j_2} > x_{i_2 , j_2} x_{i_1 , j_2} > x_{i_2 , j_2} x_{i_1 , j_1}$, hence the leading monomial of the minor $x_{i_2 , j_2} x_{i_1 , j_1} - x_{i_2 , j_1} x_{i_1 , j_2}$ is the antidiagonal term $x_{i_2 , j_1} x_{i_1 , j_2}$. A similar argument gives the desired result when $| n_1 | < | n_2 |$. 
\end{proof}

\begin{proof}[Proof of Theorem \ref{thm: main thm msv}]
By Lemma \ref{lem: iota and pi preserve schubert ideals}, we have that $\iota_m(I_{\sigma_m}) \subseteq I_{\sigma_{m + 1}}$ and $\pi_m(I_{\sigma_{m + 1}}) \subseteq I_{\sigma_{m}}$. Furthermore, if $<$ is an antidiagonal term order on $T$, let $<_m$ be the restriction of $<$ to $T_m$. Since $<_m$ is an antidiagonal term order for $T_m$ (by definition), it follows from \cite[Theorem B]{knutson2005grobner} that $T_m  /\text{in}_{<_m}(I_{\sigma_m})$ is a Noetherian Cohen-Macaulay $k$-algebra. Hence, by Theorem \ref{thm: main thm rings}, $T / \text{in}_<(\bigcup_{n \in \nn} \eta_n(I_{\sigma_n}) T)$ is a flat direct limit of the direct system $\{ T_m  /\text{in}_{<_m}(I_{\sigma_m}) , \overline{\iota}_m \}_{m \in \nn}$ of Noetherian Cohen-Macaulay $k$-algebras. Since $\bigcup_{n \in \nn} \eta_n(I_{\sigma_n}) T = I_\sigma$ by Lemma \ref{lem: partial det ideals approximate big det ideal}, the desired result follows. 
\end{proof}

\bibliographystyle{alpha}

\bibliography{references}

\end{document}